\documentclass[12pt]{article}
\usepackage[utf8]{inputenc}
\usepackage{mathtools}
\usepackage[overload]{empheq}
\usepackage{amsmath,amsfonts,amssymb,amsthm}
\usepackage{graphicx, xcolor, color, soul}
\usepackage{nameref,hyperref}
\usepackage{enumitem}
\usepackage{letltxmacro}
\usepackage{dsfont}
\usepackage{thmtools, geometry}
\usepackage[noadjust]{cite}

\theoremstyle{plain}
\newtheorem{theorem}{Theorem}[section]
\newtheorem{lem}[theorem]{Lemma}
\newtheorem{definition}{Definition}[section]
\newtheorem{heuristic}{Heuristic}

\title{On multiclass spatial birth-and-death processes with wireless-type interactions}
\author{Pierre Popineau$^{1}$, François Baccelli$^{1,2}$ \\
Affiliation: $^{1}$ INRIA-ENS, Paris, France
$^{2}$ Telecom Paris, France}
\date{}

\newcommand{\Pk}{\mathcal{P}(K)}
\newcommand{\R}{\mathbb{R}}
\newcommand{\M}{\mathcal{M}}
\newcommand{\N}{\mathbb{N}}
\newcommand{\E}{\mathbb{E}}
\newcommand{\Proba}{\mathbb{P}}
\newcommand{\D}{\mathcal{D}}
\newcommand{\diff}{\mathrm{d}}

\DeclareMathOperator*{\argmax}{arg\,max}
\newcommand{\vertiii}[1]{{\left\vert\kern-0.25ex\left\vert\kern-0.25ex\left\vert #1 
    \right\vert\kern-0.25ex\right\vert\kern-0.25ex\right\vert}}

\begin{document}

\maketitle

\begin{abstract}
	This paper studies a multiclass spatial birth-and-death (SBD) processes on a compact region of the Euclidean plane modeling wireless interactions. In this model, users arrive at a constant rate and leave at a rate function of the interference created by other users in the network. The novelty of this work lies  in the addition of service differentiation, inspired by bandwidth partitioning present in 5G networks: users are allocated a fixed number of frequency bands and only interfere with transmissions on these bands.

	The first result of the paper is the determination of the critical user arrival rate below which the system is stochastically stable, and above which it is unstable. The analysis requires symmetry assumptions which are defined in the paper. The proof for this result uses stochastic monotonicity and fluid limit models. The monotonicity allows one to bound the dynamics from above and below by two adequate discrete-state Markov jump processes, for which we obtain stability and instability results using fluid limits. This leads to a closed form expression for the critical arrival rate. The second contribution consists in two heuristics to estimate the steady-state densities of all classes of users in the network: the first one relies on a Poisson approximation of the steady-state processes. The second one uses a cavity approximation leveraging second-order moment measures, which leads to more accurate estimates of the steady-state user densities. The Poisson heuristic also gives a good estimate for the critical arrival rate.
\end{abstract}

\section{Introduction}

A key feature of wireless networks is the presence of spatial interactions between users.
More precisely, the service rate of a wireless link at any given time is a function of the interference
created by other users transmitting at the same time. This follows from Shannon's formula
(\cite{shannon48}) for the transmission rate of a wireless link when treating interference as noise.

In~\cite{sanka16}, the authors introduced a framework to study such interactions using the theory of
spatial birth-and-death processes. The setting of~\cite{sanka16} is that of 
a device to device wireless network, where users arrive at random locations in time and space and 
where their transmission rate is determined by Shannon's formula.
They derived the stability region and studied the stationary regime of this class of dynamics.

An important novelty introduced in 5G networks is the partitioning of the available bandwidth in
frequency bands of equal width and the simultaneous use of multiple frequency bands (\cite{whitepaper,bk20})
by transmissions depending on their nature: transmissions requiring lower data rates (phone calls or text messages)
are allocated less bands than those requiring higher data rates (video streaming for instance). 
This more flexible and adaptive frequency bandwidth allocation is meant to increase the network capacity.

Service differentiation based on users needs has been extensively studied, e.g., in queueing network theory. 
For instance, BCMP type queueing networks \cite{bcmp75} extend the Jackson framework to the multi-class setting.
For the class of device-to-device wireless networks introduced in~\cite{sanka16},
the multiclass paradigm remains an open problem, which is of central importance in the bandwidth part setting.
As mentioned above, in this setting, 
transmissions can use several frequency bands simultaneously.
Those requiring a higher data rate use a larger number of bands. They hence transmit with a stronger signal,
which potentially increases their transmission rate. The main novelty and difficulty is that they then
also generate a higher level of interference to other users, hence slowing down other transmissions.
Note that several new phenomena arise in this setting. For instance, users which are allocated
frequency bands that do not overlap do not directly interact. In contrast, users which are
allocated the same large set of overlapping bands and are nearby
interfere a lot and slowdown each other quite significantly, which in turn slows down all other nearby
transmissions for a longer time. 

In view of all these complex interactions, several
important questions arise. The first one is that of the stability region of this class of dynamics.
Stochastic stability has always been of central importance in communication networks,
because it ensures that e.g. user latency converges to a stationary distribution.
A second important question is that of the analytical
characterization of this stationary distribution.
For instance, in queueing network theory, the stability region of Jackson networks is
well known as well as the stationary regime \cite{kell79}. 
Another important in this service differentiation setting is that of the best allocation of resources to users in the network.
An instance of allocation policy is that which attributes radio resources proportionally to the
needs of users (see~\cite{kelly97, masso99}). Our approach is oblivious of this type of assumption.
In the present paper, arriving users in the network have their needs categorized in terms of the mean
size of the file that they have to transmit. All monotonic policies (where larger file sizes lead to larger
sets of bands) can be represented and analyzed.

The first goal of the present paper is to extend the spatial birth-and-death wireless network
framework of~\cite{sanka16} to the multiclass model described above.
As in~\cite{sanka16}, we use a birth-and-death process
of wireless dipoles (see~\cite{bacce09,haenggi16}), where dipoles arrive to a compact according to a
Poisson rain of constant intensity, representing the arrival rate of users in the network,
and leave with a departure rate proportional to the their Shannon rate in the spatial configuration.
This problem is quite complex and has been an open question for several years. 
As we will see, there is for instance no way to guess the stability condition of the
multiclass case from that of the monoclass setting derived in \cite{sanka16}.

The present paper, which is a first step in the solution of the general model,  
makes the following simplifying assumptions: (i) the traffic satisfies certain natural symmetry assumptions
(ii) the compact is assumed to be a torus, and (iii) the wireless network operates
in the low SINR case. Here is the rationale for these assumptions.
(i) Assuming a compact region of space allows one to use fluid limit techniques to asses stochastic stability.
These techniques cannot be used when the phase space is the whole Euclidean plane
(see~\cite{bacce11,sanka19} for instance). Assuming a square torus allows one to
emulate the Euclidean plane and is the simplest possible compact state space  model since 
it is invariant by translations modulo the square side.
(ii) The symmetry assumptions require that bands are in a sense exchangeable and
that users with the same needs use the same number of bands and that bands are
chosen uniformly at random. These assumptions are rather natural and are at the same
time essential in the derivations for the stability region.
(iii) the use of the linearized version of the Shannon formula, which is justified by the
low SINR regime assumption, is probably the least
essential of the three assumptions and we explain in the paper how to relax this.

The first part of the paper is focused on the stability region of such a multiclass network.
This part relies on the use of the definition of certain interference queueing networks (see~\cite{sanka19})
to obtain bounds on the stability region. Due to the nature of these dynamics, fluid limits (\cite{rybko92})
provide a natural framework to obtain conditions for stability (\cite{dai95}).

The second part uses stochastic geometry tools (see~\cite{bacce09}) as well as Palm calculus
(see~\cite{brema03}) in order to obtain heuristics for user densities in the stationary regime.
A first heuristic relies on a mean-field approximation for the interacting point processes. 
Such approximations are known to provide a powerful framework to obtain both qualitative and
quantitative results. An instance of Poisson approximation is that arising in replica mean-field systems
(see~\cite{bacce21}), where the authors prove that as the number of replicas goes to infinity,
point processes behave like independently distributed Poisson processes. In another setup,
\cite{allem22} and~\cite{gast09} provide a quantification of the error of the mean-field approximation.

\subsection*{Organization of the paper}

As mentioned earlier, we model wireless interactions using a space-time Poisson dipole model (see~\cite{sanka16} and~\cite{haenggi16}). This model and the wireless interactions we consider are defined in Section~\ref{sec:framework}. \emph{Symmetric} networks are introduced in Section~\ref{ssec:symmetric}); in such networks, users allocated the same number of communication channels have similar properties in the network, e.g., have the same arrival rate and the same file size distribution. This hypothesis allows to use combinatorics in order to obtain a closed form for the critical arrival rate in the system (through Lemma~\ref{lem:symmetric}).

The irreducibility and monotonicity properties which allow one to assess the stability region for the dynamics are gathered in Section~\ref{sec:properties}. Whereas the initial dynamic is measure-valued, monotonicity allows one to reduce the study to that of two interference queueing networks (see~\cite{sanka19}) through an adequate tessellation of space (see Section~\ref{ssec:discretization}). These two queueing networks provide an inner and an outer bound for the stability region, which yield the main stability result of this paper, stated in Theorem~\ref{thm:stability}.

To obtain the upper bound for the critical arrival rate, we rely on a fluid limit model, presented in Section~\ref{sec:fluid} and on results linking fluid limits and positive recurrence (see~\cite{shneer20} and~\cite{dai95}). To obtain instability, we use monotonicity to tie the instability of the network with the instability of an appropriate M/M/1 queue.

In the second part of the paper, we give two heuristics allowing one to quantitatively estimate the first moment of the stationary regime. These heuristics rely on stochastic geometry and rate conservation arguments. Section~\ref{ssec:poisson} discusses a first-order heuristic for the stationary user density. In Section~\ref{ssec:cavity}, we improve the quality of this heuristic by using a cavity approximation. We also compare all these heuristics with the results obtained by discrete event simulations in Section~\ref{sec:numerical}.

\section{Mathematical framework}\label{sec:framework}

\subsection{Network setup}

We consider an infrastructureless wireless network, where arrivals of transmitters follow a Poisson rain of intensity $\lambda > 0$ on a compact subset $\D$ of $\R^2$. Throughout the paper, we assume that $\D$ is a square torus.

To each arriving transmitter, we associate a receiver located at a fixed distance $r \geq 0$ in a random direction, with $r$ small compared to the side of the square. This is a Poisson dipole model as introduced in~\cite{bacce09}. We describe the configuration of users present in the network at a given time as receiver-transmitter pairs, denoted by $\Phi_t = \{(x_1, y_1), \dots (x_{N_t}, y_{N_t})\}$, where $N_t$ is the number of pairs present in the network at time $t$, $(x_i)_{1 \leq i \leq N_t}$ denotes the location of receivers and $(y_i)_{1 \leq i \leq N_t}$ that of transmitters. Let $\Phi^T_t$ be the point process describing the location of transmitters in the network and $\Phi^R_t$ be that of the locations of receivers at a given time $t$. We will refer to receiver-transmitter pairs or dipoles as \emph{users} in the rest of the paper.

$K$ orthogonal transmission bands/channels of equal width are available to users arriving in the network. Channels $i \neq j$ do not interfere with one another. Let $\Pk$ denote the set of non-empty subsets of $[1, 2, \dots, K]$. An arriving user selects a set of channels $C \in \Pk$ on which to transmit according to a given distribution $\{p_C\}_{C\in\Pk}$. This selection is made independently of the state of the network.
In the rest of the paper, we will call $C$ the \emph{class} of a user.

A receiver-transmitter pair $(x, y) \in \D$ of class $C$ arrives with a file of random size $L_{x,y}$ attached. File sizes for users in the same class $C$ are i.i.d. random variables distributed with an exponential law of mean $L_C$. Once the file has been transmitted to the receiver, the pair or dipole $(x, y)$ leaves the network.

\subsection{Symmetric system}\label{ssec:symmetric}

We say that the network is \emph{symmetric}, if all users of classes with the same cardinality have the same statistical properties in the network. In particular, whenever $\lvert C \rvert = \lvert D \rvert$, we have $p_C = p_D$ and $L_C = L_D$. 

An instance of this symmetry property is that of a system where user needs are split in $K$ categories, depending on their needs (e.g. phone calls, text messages, web browsing or video streaming). Each arriving user has a probability $p_j$ of having a need of category $j$. Upon arrival, users with needs of category $j$ are given a set of $j$ bands to transmit, sampled uniformly at random among the $\binom{K}{j}$ possibilities. Thus, the probability that a user transmits on a given set $C$ of $j$ bands is equal to $p_C = \frac{p_j}{\binom{K}{j}}$. The quantity of information this arriving user has to transmit to the network is thus sampled from an exponential distributions with parameter $L_j$ with $1 \leq j \leq K$. Setting $L_C = L_j$ and $p_C = \frac{p_j}{\binom{K}{j}}$ gives an example of a symmetric system. The symmetry in the system will be important to simplify the combinatorics we will encounter in Section~\ref{sec:stab}, when proving Theorem~\ref{thm:stability}.

The question of non-symmetric networks, where the bandwidth allocation for users allows for a differentiated use of communication channels are beyond the scope of this paper, and will be the subject of future work.

\subsection{Mathematical definitions and notation}

In this paper, vectors will be denoted as bold-faced letters and their coordinates in regular script (for instance, $\mathbf{x} = (x_i)_{0\leq i \leq d-1} \in \R^d$), and $\leq$ will denote the coordinate-wise partial ordering when used for vectors. Fluid limits associated with a stochastic process will be denoted using the same notation as the stochastic process they are derived from, in lower script, e.g. $z$ is the fluid limit associated with the process $\mathbf{Z}$.

We denote by $\M(\D)$ the set of counting measures on $\D$, all defined on a probability space $(\Omega, \mathcal{F}, \Proba)$.

For $N, K \in \N$, we denote by $D([0, \infty), \R^{N \times 2^K-1})$ the set of càdlàg functions from $\R_+$ to $\R^{N\times 2^K-1}$ and $\mathcal{B}$ is its canonical Borel $\sigma$-algebra. We consider stochastic processes as measurable maps from $(\Omega, \mathcal{F})$ to
$(D([0, \infty), \R^{N\times 2^K-1}), \mathcal{B})$.

Let $d_\infty^0$ be the infinite norm in $\R^{N\times 2^K-1}$,  $\lvert x \rvert$ be the $L^1$ norm and $\lVert x \rVert$ denote the $L^2$ norm of $x$. For a given compact set $\D \subset \R^2$, we define $\langle\cdot_\D \rangle$ as follows:

\begin{equation}\label{eq:N_D}
	\langle f_\D \rangle =
	\begin{cases}
		\int_\D f(\lVert x \rVert) \diff x, \quad \forall f : \R \rightarrow \R \\
		\int_\D f(x, 0) \diff x, \quad \forall f : \R^2 \times \R^2 \rightarrow \R,
	\end{cases}
\end{equation}

\noindent whenever the integrals are defined. We use here the same notation for two different notions; this choice is motivated by results presented in Section~\ref{sec:stab}.

For a sequence $(x^n)_{n\geq 0}$ and $x$ in $D([0, \infty), \R^{N\times 2^K-1})$, we write $x^n \rightarrow x$ if, for all $T>0$ :

\begin{equation*}
	\lim_{n\rightarrow\infty}\sup_{0\leq t \leq T}\lVert x^n(t) - x(t) \rVert = 0.
\end{equation*}

We denote by $\mathcal{E}(\R^d)$ the set of probabilities on $\R^d$ (with $d \geq 1$), and we define the partial ordering of probability measures $\leq_i$. We define:

\begin{equation*}
	I = \{f \in \mathcal{E}(\R^d), f \text{ is coordinate-wise non-decreasing}\}.
\end{equation*}

For each $F, G \in I$, we write $F \leq_i G$ if and only if

\begin{equation*}
	\int_{\R^d} f(x)F(\diff x) \leq \int_{\R^d} f(x) G(\diff x),
\end{equation*}
for all $f \in I$. $\leq_i$ is a partial order on the set of random vectors of $\R^d$. Let $\Psi$ and $\Psi'$ be two point processes on $\D$. We write $\Psi \leq \Psi'$ if we have $\Psi(\D) \leq_i \Psi'(\D)$. In the rest of the paper, we will use this partial order when comparing stochastic processes.

Finally, for two random variables $X$ and $Y$, we write $X \overset{d}{\sim} Y$ if $X$ and $Y$ have the same probability distribution. Table~\ref{tab:params} summarizes the notation used throughout the paper.

\begin{center}
    \begin{table}
        \centering
        \begin{tabular}{|@{\hspace{5pt}}l@{\hspace{5pt}}||@{\hspace{5pt}}l@{\hspace{5pt}}|} \hline
            Notation & Description \\ \hline \hline
            $\Phi^R_t, \Phi^T_t$ & Receiver (resp. transmitter) locations at time $t$ \\ \hline
            $\Phi_{C,t}$ & Point process of users of class $C$ at time $t$ \\ \hline
            $\Pk$ & Set of subsets of $[1, \dots K]$ with the exception of $\emptyset$ \\ \hline
            $\ell$ & Path-loss function \\ \hline
            $r$ & Receiver-transmitter distance \\ \hline
            $\lambda$ & Intensity of the arrival process \\ \hline
            $(p_C)_{C \in \Pk}$ & Arrival distribution of users of each type \\ \hline
            $L_C$ & Average file size for users of class $C$ \\ \hline
            $R$ & Rate-of-transmission function \\ \hline
            $\mathcal{N}_0$ & Thermal noise density in the network \\ \hline
            $\langle \ell_\D \rangle$ & $\int_{\mathbf{x} \in \D} \ell(\lVert \mathbf{x} \rVert )\diff \mathbf{x}$ \\ \hline
            $\Phi_{0,C}$ & Stationary point process of users of type $C$ \\ \hline
            $\mu_C$ & Spatial intensity of point process $\Phi_{0,C}$ \\ \hline
            $\E^0_{\Phi_{0,C}}$ & Palm expectation with respect to $\Phi_{0,C}$ \\ \hline
        \end{tabular}
        \label{tab:params}
		\vspace{10pt}
        \caption{Table of notations}
    \end{table}
\end{center}

\subsection{Wireless interactions and service times}

Let $\ell$ be a non-negative, bounded and non-increasing function with $\langle \ell_\D \rangle < \infty$ be a path-loss function. Without loss of generality, we assume that $\ell(0) = 1$.

The interference experienced by a receiver located at $x \in \D$ of class $C_x$, whose transmitter is located at $y \in \Phi_t^T$, is equal to:

\begin{equation}\label{eq:interference}
	I(x, \Phi_t) = \sum_{z \in \Phi^T_t \backslash \{y\}} \lvert C_x \cap C_z \rvert \ell(\rVert x-z \lVert),
\end{equation}

\noindent where $C_z$ is the class of the transmitter located at $z$. Note that a receiver does not interfere with its own transmitter. As mentioned earlier, the interference is a shot-noise of the transmitter point process, depending on the number of overlapping channels used by transmitters and receivers. We thus define the rate-of-transmission function for this receiver as follows:

\begin{equation}\label{eq:rof_lin}	
	R(x, \Phi_t) = \frac{A\lvert C_x \rvert \ell(r)}{\mathcal{N}_0 + I(x, \Phi_t)},
\end{equation}

\noindent where $\mathcal{N}_0 > 0$ is the thermal noise power and $A$ is a multiplicative constant. 

\eqref{eq:rof_lin} can be seen as a linearization of the Shannon-Hartley formula in the low-SINR approximation: $\lvert C_x \rvert \ell(r)$ is the signal power received by the receiver located at $x \in \D$ from its transmitter, and $\mathcal{N}_0 + I(x, \Phi_t)$ is the noise and interference power seen by this receiver. 

Under these considerations, in the low-SINR case, i.e., if $\lvert C_x \rvert \ell(r) \ll \mathcal{N}_0 + I(x, \Phi_t)$, the Shannon-Hartley (see~\cite{shannon48}) formula gives a channel capacity equal to:

\begin{equation*}
	\mathcal{C}(x) = \log_2\left(1+\frac{\lvert C_x \rvert \ell(r)}{\mathcal{N}_0 + I(x, \Phi_t)}\right) \approx \frac{1}{\ln(2)}\frac{\lvert C_x \rvert \ell(r)}{\mathcal{N}_0 + I(x, \Phi_t)},
\end{equation*}

\noindent which is the expression from~\eqref{eq:rof_lin} with $A = \frac{1}{\ln(2)}$. In the rest of the paper, we assume that $A = 1$.

The service time for a pair $(x, y)$ with a file of size $L_{x,y}$, denoted by $t_{s,x}$ is defined by:

\begin{equation}\label{eq:service_time}
	t_{s,x} = \min\left\{w > t_x : \int_{t_x}^w R(x, \Phi_u) \diff u > L_{x,y} \right\}.
\end{equation}

The instantaneous departure rate of a receiver-transmitter pair of class $C$ with receiver located at $x \in \D$, at time $t$, is equal to:

\begin{equation}\label{eq:death_rate}
	d(x, t) = \frac{1}{L_C}R(x, \Phi_t).
\end{equation}

\section{Preliminary properties and definitions}\label{sec:properties}

We establish a stochastic recurrence for the process $\Phi_t$, which will help us asserting a central property for the network, is \emph{stochastic monotonicity}. 

A random variable $X$ with values in $\R^d$ \emph{dominates} $Y$ if the cumulative distribution functions of $X$ and $Y$, denoted as $F_X$ and $F_Y$, are such that $F_Y \leq_i F_X$. A consequence of stochastic domination is, for all $\mathbf{x} \in \R^d$ and for all probabilities $\Proba$ defined on $\R^d$, we have $\Proba[X > \mathbf{x}] \geq \Proba[Y > \mathbf{x}]$.

We can define a stochastic recurrence relation for each of the subprocesses $\Phi_{C,t}$: denote as $(T_{C,n})_{n \leq 0}$ the times at which events (either an arrival or a departure) happen in $\Phi_{C,t}$. The relation between $\Phi_{C,T_{C,n}}$ and $\Phi_{C, T_{C,n+1}}$ is given by:

\begin{equation}\label{eq:sto_rec}
	\begin{cases}		
		\Phi_{C, T_{C,n+1}} &= \Phi_{C, T_{C,n}} + \mathcal{B}_{C, n} - \mathcal{D}_{C, n}, \\
		T_{C,n+1} &= T_{C,n} + \delta_{C,n},
	\end{cases}
\end{equation}

\noindent where $\mathcal{B}_{C,n}$ and $\mathcal{D}_{C,n}$ are $\mathcal{M}(\D)$-valued random variables that are defined as follows: let $b_C$ be a real-valued random variable distributed according to an exponential distribution with rate $\lambda p_C$ and $\mathbf{d}_{C,n}$ be a $\R^{\Phi_{C,n}(\D)}$-valued random variable where each coordinate is distributed according an independent exponential distribution with mean $\frac{1}{L_C} R(x_i, \Phi_{T_{C,n}})$, for each $x_i \in \Phi_{C, T_{C,n}}$. Then, the values of $\mathcal{B}_{C,n}$ and $\mathcal{D}_{C,n}$ are given as follows :
\begin{itemize}
	\item[-] $\mathcal{B}_{C,n}$ is equal to the null measure if $b_C \geq \min \mathbf{d}_{C,n}$, or equal to $\delta_x$, where $x$ is sampled uniformly in $\D$ otherwise;
	\item[-] $\mathcal{D}_{C,n}$ is equal to the null measure if $b_C \leq \min \mathbf{d}_{C,n}$ or if $\Phi_{C,n}(\D) = 0$, or equal to $\delta_{x_i}$ where $i$ is the smallest coordinate of $\mathbf{d}_{C,n}$ otherwise.
\end{itemize}
Finally, the time between events is equal to $\delta_{C,n} = \min\{b_C, \min \mathbf{d}_{C,n}\}$. More results about partial ordering of stochastic recurrences can be found in Chapter 4 of~\cite{brema03}.

\subsection{Stochastic monotonicity}

Stochastic monotonicity will help us obtaining bounds for the dynamics of the network by defining queueing systems that dominate or are dominated by our dynamics. We need the following theorem:

\begin{theorem}\label{thm:domination}
	Let $\Delta : (\Phi_0, \lambda, L, R) \mapsto \Phi_t$ define the realization of the dynamics, where $\Phi_0$ is the initial condition of the network, $\lambda$ is the arrival rate, $\mathbf{L}$ is the vector of average file sizes and $R$ is the rate-of-transmission function. Let $\Phi = \Delta(\Phi_0, \lambda, d)$ and $\Phi' = \Delta(\Phi'_0, \lambda', d')$. The following conditions are sufficient for $\Phi'$ to dominate $\Phi$ (with all the other parameters taken equal) :
	\begin{itemize}
		\item[$i)$] $\lambda \leq \lambda'$;
		\item[$ii)$] $\mathbf{L} \leq \mathbf{L}'$;
		\item[$iii)$] for all point processes $\Psi \leq \Psi'$ on $\D$ and $x \in \D$, $R'(x, \Psi') \leq R(x, \Psi)$;
		\item[$iv)$] $\Phi_0 \subseteq \Phi'_0$.
	\end{itemize}
\end{theorem}

It is to note that condition $iii)$ is met when we have two path-loss functions $\ell$ and $\ell'$ such that for all $r \geq 0$, $\ell(r) \leq \ell'(r)$.

\begin{proof}

	To obtain domination, we use a coupling argument between two instances of the dynamics to obtain the domination relation, similar as the ones used in Appendix B of~\cite{sanka19}.

	Let us take $0 < \lambda < \lambda'$ with the same fixed initial condition $\Phi_0$, a vector $\mathbf{L}$ and a rate-of-transmission function $R$, and let $\Phi_t = \Delta(\Phi_0, \lambda, d)$ and $\Phi'_t = \Delta(\Phi_0, \lambda', d)$. The quantities related to $\Phi'$ will be denoted with a prime. We want to prove that for all times $0 \leq t$, $\Phi_t \subseteq \Phi'_t$.
	
	We couple both the arrival and the departure processes in the network as follows: the arrival process of $\Phi$ is a Poisson rain $\mathcal{A}$ with parameter $\lambda$, and the arrival process for $\Phi'$ is $\mathcal{A} \cup \mathcal{A}'$, with $\mathcal{A}'$ being a Poisson rain with intensity $\lambda' - \lambda > 0$ independent from $\mathcal{A}$, so that common arrivals in $\Phi_t$ and $\Phi'_t$ happen at the same locations and times. Using this coupling, we will show that, at all times $t$, $\Phi_t \subseteq \Phi'_t$.
	
	At time $t = 0$, we have trivially $\Phi_0 \subseteq \Phi_0$. At time $t > 0$, we know that both $\Phi_t(\D)$ and $\Phi'_t(\D)$ are almost surely finite. Let us write $\Phi_t = \sum_{C \in \Pk} \sum_{i = 1}^{\Phi_{C,t}(\D)} \delta_{x_{C,i}}$ and $\Phi'_t = \sum_C \sum_{i = 1}^{\Phi'_{C,t}(\D)} \delta_{x'_{C,i}}$.
	
	Assume that up to time $t$, $\Phi_t \subseteq \Phi'_t$, the next event happening at a time $\hat{t} \geq t$ can be one of the following nature:

	\begin{itemize}
		\item[-] The arrival of a user of class $C$ in $\Phi_t$;
		\item[-] The arrival of a user of class $C$ in $\Phi'_t$;
		\item[-] The departure of a user of class $C$ in $\Phi_t$;
		\item[-] The departure of a user of class $C$ in $\Phi'_t$.  
	\end{itemize}
	
	Because arrivals are coupled, an arrival in either $\Phi_t$ or $\Phi'_t$ maintains the inclusion. The same holds if an element of $\Phi_t$ leaves. The last case to look at is the departure of an element $x$ of $\Phi'_t$.

	Let $\hat{t}^-$ be such that $\hat{t}^- = \hat{t}$ but the departure of $x$ has not happened yet. From our assumptions, we know that $\Phi_{D, \hat{t}^-} \subseteq \Phi'_{D, \hat{t}^-}$ for all classes $D \in \Pk$. Assume the next departure is for a user at $x$ of class $C \in \Pk$:
	
	\begin{align*}
		R(x, \Phi_{\hat{t}^-}) & = \frac{\lvert C_x \rvert \ell(r)}{\mathcal{N}_0 + \sum_{y \in \Phi_{\hat{t}^-}} \lvert C_x \cap C_y \rvert \ell(\lVert x - y \rVert)} \\
		& \overset{(a)}{\geq} \frac{\lvert C_x \rvert \ell(r)}{\mathcal{N}_0 + \sum_{y \in \Phi'_{\hat{t}^-}} \lvert C_x \cap C_y \ell(\lVert x - y \rVert)} \\
		& = R(x, \Phi'_{\hat{t}^-}),
	\end{align*}

	\noindent where $(a)$ uses the fact that $\lvert \Phi_t(\D) \rvert \leq \lvert \Phi'_t(\D) \rvert$. In other words, the departure rates in $\Phi_{\hat{t}^-}$ are larger than in $\Phi'_{\hat{t}^-}$. 
	
	We take the \emph{Poisson imbedding} of the departure processes (see~\cite{brema96}): let $\mathcal{D}_{C,t}$ and $\mathcal{D}'_{C,t}$ be the point processes of users of class $C$ that left each system up to time $t$. These processes have respective stochastic intensities $\frac{1}{L_C}R(x, \Phi_t)$ and $\frac{1}{L_C}R(x, \Phi'_t)$. Using Lemma 3 from~\cite{brema96}, we can imbed them on the same Poisson point process $\mathcal{N}$ of intensity 1 on $\R^2$. 
	
	Using this Poisson imbedding, any point $x$ leaving $\Phi'_{C, \hat{t}^-}$ has already left $\Phi_{C, \hat{t}^-}$, which proves that the inclusion is maintained if the next event to come is a departure in $\Phi'_t$ and this concludes the proof of Condition $i)$.

	To obtain the other conditions, we use the same argument and we compare the rate-of-transmission functions in each case to get the required inclusion.
\end{proof}

Theorem~\ref{thm:domination} is central to our study, because it will allow us to bound from above and below the dynamics of the network we are studying in order to obtain bounds for the limits of the stability region in the network.

\subsection{Irreducibility}

The second important property about the process $\Phi_t$ is the following: 

\begin{theorem}\label{thm:phi_irreducible}
	$\Phi_t$ is a $\phi$-irreducible Markov jump process on $\M(\D)$. 
\end{theorem}

It is to note that the similarity in notation between $\phi$-irreducibility and the point process $\Phi_t$ is coincidental, the two notions are independent.

\begin{proof}
	Our goal is to define a measure on the set of counting measures $\M(\D)$. We start by letting $\phi$ be the set function defined as follows:

	\begin{itemize}
		\item $\phi(\{0_{\M(\D)}\}) = 1$, where $0_{\M(\D)}$ is the measure associated with the null counting measure,
		\item For $k \in \N$ and $B_1, B_2, \dots, B_N$ disjoint Borel subsets of $\D$, we define the event $A_n = \{\Phi \in \M(\D) : \Phi(\D) = n, \Phi = (y_1, y_2, \dots, y_n) \in B_1 \times \dots \times B_n, \Phi(B_1)>0, \dots \Phi(B_n) > 0\}$. We then set:
		\begin{equation*}
			\phi(A_n) = \frac{1}{2^n}H(B_1)H(B_2) \dots H(B_n),
		\end{equation*}
		where $H$ is a Haar measure on the square torus $\D$.
	\end{itemize}

	Here, $\phi$ is a set function on the \emph{semiring} of sets $(A_n)_{n \in \N, (B_i)\in \mathcal{F}^n}$. Using Theorem 11.3 from~\cite{billi99}, we can extend $\phi$ in a unique way to a measure on the $\sigma$-field generated by the set of events $\{A_n : n \in \N, (B_i)\in \mathcal{F}^n \}$, which is equal to $\mathcal{F}$. For ease of notation, we will denote this measure by $\phi$ as well. This way, $\phi$ is a measure defined on the set of counting measures $\mathcal{M}(\D)$.

	To obtain $\phi$-irreducibility, we will proceed in two steps: in the first step, we prove that, from each state $\Phi$ with $N$ points, we can reach the empty state with positive probability, and in the second step, we prove that, from the empty state, we can reach any state $A$ with $\Phi(A) > 0$ with positive probability in a finite number of steps, which allows us to conclude.

	Let us assume that at a given time $t$, $N$ receiver-transmitter pairs are present in the system, with locations $(x_i, y_i)_{1 \leq i \leq N}$, $x_i \in \Phi_t^R$ and $y_i \in \Phi_t^T$, and let $(C_i)_{1\leq i \leq N}$ be the classes of these users. The probability that we have the departure of $(x_1, y_1), \dots, (x_n, y_n)$ in the next $N$ steps is equal to:

	\begin{equation*}
		\Proba[N \text{ departures in a row} \left\lvert \Phi_t \right.] = \prod_{i = 1}^N \frac{d_{C_i}}{d_{C_i} + b_{C_i}},
	\end{equation*}

	\noindent where $b_{C_i} = \lambda p_{C_i} \lvert \D \rvert$ and $d_{C_i} = \frac{1}{L_{C_i}}R(x_i, \Phi_{t_i})$, with $\Phi_{t_i}$ being the network configuration after the departure of the user $(x_i, y_i)$. This probability is non-null, i.e., we can reach the null measure with nonzero probability from any other state. 
	
	Similarly, assume we start from the empty state, let $n$ be an integer, $B_1, \dots, B_n$ be $n$ disjoint Borel subsets of $\D$, and let us define $A_n$ as previously. By definition, $\Phi(A_n) > 0$. The probability to reach $A_n$ from the empty measure $\emptyset$ can be expressed as:

	\begin{multline*}
		\Proba^n \left(\emptyset, A_n\right) = \int_{y_1 \in B_1, \dots, y_n \in B_n} \int_{t_1 = 0}^\infty \dots \int_{t_n = t_{n-1}}^\infty \mspace{-12mu} H(\diff y_1)\dots H(\diff y_n) e^{-\lambda p_{C_1} t_1} \dots e^{-\lambda p_{C_n}(t_n-t_{n-1})} \prod_i \frac{b_{C_i}}{b_{C_i} + d_{C_i}} \lambda p_{C_1}\diff t_1 \dots \lambda p_{C_n}\diff t_n,
\end{multline*}

	\noindent where $C_i$ denotes the class of the $i$-th user. This probability is positive, which means that we can reach $A_n$ from the empty measure with positive probability in $n$ steps, which gives us $\phi$-irreducibility for the process $\Phi_t$.
\end{proof}

\subsection{Stability}

A first question of interest is that of the stability region. In the rest of the paper, we use the following definition for stability:

\begin{definition}
	An SBD process is said to be \emph{stable} if the Markov chain $\Phi_t$ is positive recurrent, and \emph{unstable} if it is null recurrent or transient.
\end{definition}

We can now state the following theorem:

\begin{theorem}\label{thm:cutoff}
	Under the foregoing assumptions, there exist two values $0 \leq \lambda_c^- \leq \lambda_c^+$ such that: 
	\begin{itemize}
		\item $\forall \lambda < \lambda_c^-$, $\Phi_t$ is stable;
		\item $\forall \lambda > \lambda_c^+$, $\Phi_t$ is unstable.
	\end{itemize}
\end{theorem}

\begin{proof}

	This theorem is a consequence of Theorem~\ref{thm:domination}: if the system is stable for a given $\lambda_0$, then it is stable for each $\lambda < \lambda_0$ and if the system is unstable for a given $\lambda_1$, then it is unstable for all $\lambda > \lambda_1$. The existence of cutoff (or critical) arrival rates is well-established in monotonic queueing networks, and was adapted to spatial birth-and-death processes (see Theorems 1 and 2 of~\cite{sanka16} for instance). Let $\lambda_c^-$ and $\lambda_c^+$ be defined as:

	\begin{align*}
		\lambda_c^- & = \sup_{\lambda > 0}\{\lambda \text{ such that } \Phi_t \text{ is stable}\}, \\
		\lambda_c^+ & = \inf_{\lambda > 0}\{\lambda \text{ such that } \Phi_t \text{ is unstable}\}.
	\end{align*}

	$\Phi_t$ is an irreducible Markov chain on $\mathcal{M}(\D)$. We know that the Markov chain is either recurrent or transient (see Theorem 8.2.5 from~\cite{meyn12}). Using stochastic monotonicity in the network from Theorem~\ref{thm:domination}, we know that the chain $\Phi_t$ is positive recurrent for all $\lambda \leq \lambda_c^-$ and transient for all $\lambda \geq \lambda_c^+$, which gives $\lambda_c^- \leq \lambda_c^+$. Since the system is trivially stable for $\lambda = 0$ (in which case $\Phi_t$ is constantly equal to the null measure), $\lambda_c^-$ and $\lambda_c^+$ are well defined and $\lambda_c^- \geq 0$, which concludes the proof.

\end{proof}

We can refine this result by obtaining an upper and a lower bound bounds for $\lambda_c$:

\begin{lem}\label{lem:bounds_1}
	$\lambda_c^-$ and $\lambda_c^+$ satisfy:
\begin{equation*}
	\frac{\ell(r)}{K \bar{L} \langle \ell_\D \rangle} \leq \lambda_c^- \leq \lambda_c^+ \leq \frac{K\ell(r)}{\underline{L} \langle \ell_\D \rangle},
\end{equation*}

\noindent where $\bar{L} = \max_C L_C$, $\underline{L} = \min_C L_C$ and $\langle \ell_\D \rangle = \int_{x\in \D} \ell(\lVert x \rVert) \diff x$.
\end{lem}

\begin{proof}
The rate-of-transmission is bounded from below by that where all interfering users transmit on all channels (i.e., are of class $[1,\dots, K]$) and the transmitting user uses only one channel, i.e., for all $x, t$:

 \begin{equation*}
 	R(x,\Phi_t) \geq R_u(x, \Phi_t) \triangleq \frac{\ell(r)}{\mathcal{N}_0 + \sum_{y \in \Phi_t \backslash \{x\}}K\ell(\lVert x-y \rVert)}.
 \end{equation*}

Moreover, by definition, $\mathbf{L} < \bar{L}$ component-wise.

Conversely, the rate-of-transmission is bounded from above by that where interfering users use a single channel and the transmitting users use all $K$ channels, i.e., for all $x, t$:

\begin{equation*}
R(x,\Phi_t) \leq R_d(x, \Phi_t) \triangleq \frac{K\ell(r)}{\mathcal{N}_0 + \sum_{y \in \Phi_t \backslash \{x\}}\ell(\lVert x-y \rVert)},
\end{equation*}

\noindent and $\underline{L} \leq \mathbf{L}$ component-wise. For any given initial condition $\Phi_0$ and arrival rate $\lambda$, let $\Phi_{u,t} = \Delta(\Phi_0, \lambda, \underline{L}, R_u)$ and $\Phi_{d,t} = \Delta(\Phi_0, \lambda, \bar{L}, R_d)$. $\Phi_u$ and $\Phi_d$ two monotype dipolar Poisson networks, as defined in~\cite{sanka16}. Using the main result of~\cite{sanka16}, we know that the cut-off arrival rate is equal to $\frac{\ell(r)}{K \bar{L} \langle \ell_\D \rangle}$ for $\Phi_{d,t}$ and to $\frac{K\ell(r)}{\underline{L} \langle \ell_\D \rangle}$ for $\Phi_{u,t}$. 

We apply Theorem~\ref{thm:domination}, which states that $\Phi_{u,t}$ dominates $\Phi_t$ and $\Phi_{d,t}$ is dominated by $\Phi_t$ to obtain the intended inequality. Since $\frac{\ell(r)}{K \bar{L} \langle \ell_\D \rangle}> 0$ and $\frac{K\ell(r)}{\underline{L} \langle \ell_\D \rangle} < \infty$, we can conclude that $0 < \lambda_c^-$ and  $\lambda_c^+ < \infty$.
\end{proof}

\section{Main results}\label{sec:stab}

In this section, we study the stability of the system through stochastic domination and the use of a discretization of the dynamics of the SBD process. We state the following theorem:

\begin{theorem}\label{thm:stability}

	In the symmetric setup, $\lambda_c^- = \lambda_c^+ \triangleq \lambda_c$, where $\lambda_c$ is the \emph{critical} arrival rate, equal to:

	\begin{equation}
		\lambda_c = \frac{K \ell(r)}{\langle \ell_\D \rangle \mathfrak{L}},
	\end{equation}

	\noindent where $\mathfrak{L} \triangleq \sum_C p_C \lvert C \rvert L_C$.
	
\end{theorem}

The proof for this result happens in three steps. In a first step, we define a discretization of the dynamics to reduce the study of the jump process $\Phi_t$ to the study of two adequate queueing networks, one that bound the dynamics from above and another on that bounds the dynamics from below. We then move on to proving the stability of the former network and the instability of the latter to obtain two bounds for the critical arrival rate of the network, which coincide to the value of $\lambda_c$.

\subsection{Discretization of the dynamics}\label{ssec:discretization}

To study our dynamics, we introduce the following discretization: for $\varepsilon > 0$, we tessellate $\D$ in $N_\varepsilon$ square cells of side length $\varepsilon$, such that the origin is at the center of its cell. We denote by $A_i$ the cell centered at $a_i \in \D$. Finally, we introduce the stochastic process $\mathbf{X}_\varepsilon(t) = (X_{i,C}(t))_{0\leq i \leq N_\varepsilon -1, C \in \Pk}$, where for all $i, C$, $X_{i,C}(t) = \Phi_{t,C}(A_i)$ is the number of receiver-transmitter pairs of type $C$ in cell $i$. 

We now use this discretization of the dynamics to define two \emph{interference queueing networks} (as defined in~\cite{sanka19}) $\bar{\mathbf{X}}_{\varepsilon}$ and $\underline{\mathbf{X}}_{\varepsilon}$ with state space $\N^{N_\varepsilon \times 2^K-1}$ such that:

\begin{itemize}
	\item $\bar{\mathbf{X}}_{\varepsilon}$ dominates $\mathbf{X}_\varepsilon$
	\item $\underline{\mathbf{X}}_{\varepsilon}$ is dominated by $\mathbf{X}_\varepsilon$.
\end{itemize}

Let $\bar{X}_{i,C}(t)$ and $\underline{X}_{i,C}(t)$ denote the respective number of users of type $C$ in cell $i$ in each of the two processes. For a given configuration of each process, we define the dynamics as follows:

\begin{itemize}
	\item The birth process of $\bar{\mathbf{X}}_\varepsilon$ is a Poisson rain of intensity $\lambda \lvert \D \rvert$ and the death rate for users in cell $i$ of class $C$ at time $t$ is $\frac{\bar{X}_{i,C}(t)}{L_C}\bar{R}_{i,C}(t)$
	\item The birth process of $\underline{\mathbf{X}}_\varepsilon$ is a Poisson rain of intensity $\lambda \lvert \D \rvert$ and the death rate for users in cell $i$ and class $C$ at time $t$ is $\frac{\underline{X}_{i,C}(t)}{L_C}\underline{R}_{i,C}(t)$.
\end{itemize}

To define the functions $\bar{R}_{i,C}$ and $\underline{R}_{i,C}$, we start by defining two path-loss functions $\ell^\varepsilon$ and $\ell_\varepsilon$ as follows:

\begin{equation*}
	\begin{cases}
		\ell^\varepsilon(x, y) = \ell^\varepsilon(a_i, a_j) \hfill \text{ if } x \in A_i, y \in A_j\\
		\ell^\varepsilon(a_i, a_j) = \max\left\{\ell(\lVert b_i - b_j\rVert), b_j \in \mathcal{V}_j, b_i \in \mathcal{V}_i \right\},&
	\end{cases}	
\end{equation*}

\noindent and

\begin{equation*}
	\begin{cases}
		\ell_\varepsilon(x, y) = \ell_\varepsilon(a_i, a_j) \hfill \text{ if } x \in A_i, y \in A_j\\
		\ell_\varepsilon(a_i, a_j) = \min\left\{\ell(\lVert b_i - b_j\rVert), b_j \in \mathcal{V}_j, b_i \in \mathcal{V}_i \right\},
	\end{cases}
\end{equation*}

\noindent where, for $0 \leq i \leq N_\varepsilon - 1$, we define $\mathcal{V}_i = \left\{b_i, \lVert b_i - a_i \rVert \in \{0,\varepsilon\}\right\}$. Using the dominated convergence theorem, we get:

\begin{equation}\label{eq:integral}
	\lim_{\varepsilon \rightarrow 0^+} \langle \ell^\varepsilon_\D \rangle = \lim_{\varepsilon \rightarrow 0^+} \langle \ell_{\varepsilon,\D} \rangle = \langle \ell_\D \rangle,
\end{equation}

\noindent where $\langle\cdot_\D \rangle$ is defined in~\eqref{eq:N_D}. Because of the square torus topology of $\D$, we have, for all $i$:

\begin{equation}\label{eq:ell_epsilon}
\sum_{k=0}^{N_\varepsilon -1} \ell^\varepsilon(a_i, a_k) = \sum_{k=0}^{N_\varepsilon -1} \ell^\varepsilon(0, a_k) = \frac{1}{\varepsilon^2} \langle \ell^\varepsilon_\D \rangle.
\end{equation}

The same result holds for $\ell_\varepsilon$.

The interferences $\bar{I}_{i,C}$ and $\underline{I}_{i,C}$ experienced by users in cell $i$ of class $C$ and with path-loss functions $\ell^\varepsilon$ and $\ell_\varepsilon$ are respectively equal to:

\begin{align*}
\bar{I}_{i,C}(t) &= \sum_{k, U}\lvert C \cap U \rvert \ell^\varepsilon(a_k, a_i) \left(\bar{X}_{k,U}(t) - \mathbf{1}_{\{U = C, i = k\}}\right) \\
\underline{I}_{i,C}(t) &= \sum_{k, U}\lvert C \cap U \rvert \ell^\varepsilon(a_k, a_i) \left(\underline{X}_{k,U}(t) - \mathbf{1}_{\{U = C, i = k\}}\right).
\end{align*}

Finally, the rate-of-transmission functions of users in cell $i$ of class $C$ are defined by:

\begin{align*}
\bar{R}_{i,C}(t) &= \frac{\lvert C \rvert}{\mathcal{N}_0 + \bar{I}_{i,C}(t)} \\
\underline{R}_{i,C}(t) &= \frac{\lvert C \rvert}{\mathcal{N}_0 + \underline{I}_{i,C}(t)}.
\end{align*}

\begin{theorem}\label{thm:up_down}
	$\bar{\mathbf{X}}_\varepsilon$ and $\underline{\mathbf{X}}_\varepsilon$ are two irreducible Markov jump processes with state space $\N^{N_\varepsilon \times 2^K-1}$. Moreover, $\bar{\mathbf{X}}_\varepsilon$ stochastically dominates $\mathbf{X}_\varepsilon$, and $\underline{\mathbf{X}}_\varepsilon$ is stochastically dominated by $\mathbf{X}_\varepsilon$.
\end{theorem}

\begin{proof}
	From the definition of the function $\ell^\varepsilon$ and $\ell_\varepsilon$, we know that for all $x, y \in \D$, we have:

	\begin{equation*}
		\ell_\varepsilon(x,y) \leq \ell(\lVert x - y \rVert) \leq \ell^\varepsilon(x,y).
	\end{equation*}

	From this definition, we can obtain that for all $C \in \Pk$, $0 \leq i \leq N_\varepsilon - 1$, $x \in A_{i,C}$ and a given network configuration $\Phi_t$, we have:

	\begin{equation*}
		\bar{R}_{i,C}(t) \leq R(x, \Phi_t) \leq \underline{R}_{i,C}(t).
	\end{equation*}

	To obtain the domination relation, we can now apply Theorem~\ref{thm:domination}. Irreducibility is obtained by developing the same argument as for Theorem~\ref{thm:phi_irreducible}, which concludes the proof.
\end{proof}

Using Theorem~\ref{thm:up_down}, we can now reduce the study of the SBD process $\Phi_t$ to the study of the stability of the two processes $\underline{\mathbf{X}}_\varepsilon$ and $\bar{\mathbf{X}}_\varepsilon$. In the next section, we introduce a framework to obtain a condition on the arrival rate such that the former is unstable, and the latter is stable.

\subsection{Fluid limits and fluid model}\label{sec:fluid}

From the definitions of the dynamics in our system, we can establish an equation for the dynamics of $\bar{\mathbf{X}}_\varepsilon$, which will allow us to obtain a condition on $\lambda_c^-$ for the stability of the system.

Let $(\mathcal{A}_{i,C})$ and $(N_{i,C})$ be two families of independent Poisson processes with intensity 1. By definition of the dynamics of the discretized systems, arrivals of users in cell $i$ of class $C$ happen with rate $\lambda p_C \varepsilon^2$. If the system is in state $x = (x_{i,C}) \in \N^{N_\varepsilon \times 2^K-1}$ at time $t$, users in cell $i$ and class $C$ have a departure rate equal to $\frac{1}{L_C}x_{i,C}(t)\bar{R}_{i,C}(t)$. We thus obtain the following equation ruling the evolution of the population in the chain $\bar{\mathbf{X}}_\varepsilon$:

\begin{align}
	\bar{X}_{i,C}(t) = &\bar{X}_{i,C}(0) + \mathcal{A}_{i,C}(\lambda p_C \varepsilon^2t) \nonumber \\
	&- N_{i,C}\left( \frac{1}{L_C}\int_0^t \bar{X}_{i,C}(u)\bar{R}_{i,C}(u)\diff u\right).\label{eq:dynamic}
\end{align}

The goal of this section is to find a condition on $\lambda$ so that the Markov chain $\bar{\mathbf{X}}_\varepsilon(t)$ is positive recurrent. Using~\eqref{eq:dynamic}, we get the following theorem for fluid limits in our discretized network:

\begin{theorem}\label{thm:FSP}

	The fluid limits for the chain $\bar{\mathbf{X}}_\varepsilon$ exist and are solutions to the following system of equations:

	\begin{equation*}
		\begin{cases}
			\bar{x}'_{i,C}(t) = \lambda p_C \varepsilon^2 - \frac{1}{L_C}\frac{\lvert C\rvert \bar{x}_{i,C}(t)}{\sum_{k, U} \lvert C \cap U \rvert \ell^\varepsilon(a_k, a_i) \bar{x}_{k,U}} \hfill \textrm{ if } \bar{\mathbf{x}}(t) \neq 0 \\
			\bar{\mathbf{x}}_0 = \bar{\mathbf{X}}_\varepsilon(0)
		\end{cases}.
	\end{equation*}
	
\end{theorem}

The proof for this result is detailed in Appendix A. It is to note that these dynamics are not well-defined in the case $\bar{\mathbf{x}}(t) \neq 0$, but we do not need such a definition: the results about stability and fluid limits we use involve fluid limits with a strictly positive initial condition. The structure of this proof is adapted from~\cite{reed14}, where the authors propose a method to prove the existence and uniqueness of fluid limits for a given range of dynamics. Fluid limits prove to be powerful tools in order to study the stability and instability of queueing network.

\subsection{Stability for symmetric multiclass wireless networks}

We start by studying the stability of the chain $\bar{\mathbf{X}}_\varepsilon$, which gives us an lower bound for $\lambda_c^-$. We have the following theorem:

\begin{theorem}\label{thm:forward}
	Let $\varepsilon > 0$ and $\bar{\lambda}_\varepsilon = \frac{K}{\mathfrak{L}\langle \ell^\varepsilon_\D \rangle}$. If $\lambda < \bar{\lambda}_\varepsilon$, then the chain $\bar{\mathbf{X}}_\varepsilon$ is stable.
\end{theorem}

The proof for this result, relying on~\cite{shneer20}, is given in Appendix B.

We know from Theorem~\ref{thm:up_down} that the process $\bar{\mathbf{X}}_\varepsilon$ dominates the original dynamic. From Theorem~\ref{thm:domination}, we obtain that, for all $\varepsilon > 0$, $\lambda_c \geq \frac{K}{\mathfrak{L}\langle \ell^\varepsilon_\D \rangle}$. Taking the limit as $\varepsilon$ goes to 0 gives us:

\begin{equation*}
	\lambda_c^- \geq \frac{K}{\mathfrak{L}\langle \ell_\D \rangle}.
\end{equation*}

The stability of the fluid scaled model is only a necessary condition for the stability of the Markov chain and not a sufficient one: in~\cite{bramso99}, we have an example of a stable queueing system with an unstable fluid model. An idea to prove instability in the network would be to rely on the concept of \emph{weak instability}, (see Definition 4.1 from~\cite{dai96}). Theorem 4.2 from the same paper provides a result linking weak instability of a fluid limit model and the transience of the associated Markov chain.

Unfortunately, to use this result, we have to study fluid limits starting from 0, and the dynamics of our fluid limit model are not defined when $\bar{\mathbf{x}} = 0$. To obtain instability for the system, we will use stochastic domination and an adequate Markov chain to bound $\underline{\mathbf{X}}$ from below. We have the following theorem:

\begin{theorem}\label{thm:reciprocal}
	Let $\underline{\lambda}_\varepsilon = \frac{K}{\langle \ell_{\varepsilon,\D} \rangle \mathfrak{L}}$. In the symmetric case, if $\lambda > \underline{\lambda}_\varepsilon$, then $\underline{\mathbf{X}}_\varepsilon$ is unstable.
\end{theorem}

The proof of this theorem is presented in Appendix C. We know that the network is unstable if $\lambda > \frac{K}{\langle \ell_{\varepsilon,\D} \rangle \mathfrak{L}}$ for each $\varepsilon > 0$. Taking the limit as $\varepsilon$ goes to 0, we conclude:

\begin{equation*}
	\lambda_c^+ \leq \frac{K}{\langle \ell_\D \rangle \mathfrak{L}}.
\end{equation*}

When combining the results of Theorem~\ref{thm:forward} and Theorem~\ref{thm:reciprocal}, we obtain that:

\begin{equation*}
	\lambda_c^- = \lambda_c^+ = \frac{K}{\langle \ell_\D \rangle \mathfrak{L}},
\end{equation*}

\noindent which concludes the proof of Theorem~\ref{thm:stability} in the case $r = 0$.

\subsection{Generalization to $r>0$}

To generalize the result of Theorem~\ref{thm:stability} to an arbitrary link length $r > 0$, we have to consider both the location of the transmitters and of the receivers in the system, because a transmitter will not necessarily located in the same cell as its receiver.

Let $\bar{\mathbf{Z}}_\varepsilon(t) = (Z_{i,C}(t))$ be the $N_\varepsilon \times 2^K-1$ matrix of transmitter locations in the dominating system and $\bar{\mathbf{M}}_\varepsilon(t)$ be such that $\bar{M}_{i,C}(t)$ is a $N_\varepsilon \times 2^K-1$ matrix whose coordinate $(j, D)$ denotes how many transmitters in cell $j$ of class $D$ have a receiver in cell $i$ of class $C$ at time $t$ in the network. 

The process $\bar{\mathbf{S}}(t) = (\bar{\mathbf{X}}_\varepsilon(t), \bar{\mathbf{Z}}_\varepsilon(t), \bar{\mathbf{M}}_\varepsilon(t))$ is a Markov chain with countable state space. The evolution of $\bar{\mathbf{S}}(t)$ is the following: to each receiver of class $C$ arriving in cell $i$ with rate $\lambda p_C \varepsilon^2$, we pick uniformly a point $x$ in the cell $A_i$. We then draw a circle with radius $r$ centered at $x$ and we pick a point $y$ uniformly at random on the circle, which decides the cell in which the transmitter is located. To obtain the interference seen in the network, we sum over the locations of the transmitters, i.e. the interference from~\eqref{eq:interference} involves the vector $\bar{\mathbf{Z}}_\varepsilon(t)$. When a receiver leaves the system, we delete uniformly at random a transmitter such that its receiver is located in the cell where the departure happened.

Using the definitions of the process $\bar{\mathbf{S}}(t)$ and the same steps as for Theorem~\ref{thm:fluid_equation}, we can obtain fluid limits for the process $\bar{\mathbf{X}}_\varepsilon(t)$ in this setup (which is symmetric):

\begin{equation*}
	\bar{x}'_{i,C}(t) = \lambda p_C \varepsilon^2 - \frac{1}{L_C}\frac{\lvert C \rvert \ell(r) \bar{x}_{i,C}(t)}{\sum_{k,U}\ell^\varepsilon(a_k, a_i) \lvert C \cap U \rvert \bar{z}_{k,U}(t)}.
\end{equation*}

For all $i, C$ and $t \geq 0$, let $i^\star_C(t)$ be such that $z_{i^\star_C(t), C} = \max_i z_{i,C}(t)$. The process $(\bar{z}_{i^\star_C(t), C})$ dominates $\bar{\mathbf{Z}}_\varepsilon$, and we have:

\begin{equation*}
	\bar{z}'_{i,C}(t) \leq \lambda p_C \varepsilon^2 - \frac{1}{L_C}\frac{\lvert C \rvert \ell(r) \bar{z}_{i,C}(t)}{\sum_{k,U}\ell^\varepsilon(a_k, a_i) \lvert C \cap U \rvert \bar{z}_{k^\star_U(t),U}(t)}.
\end{equation*}

We can now use similar arguments as for Theorem~\ref{thm:forward} to obtain the stability of the Markov process $\bar{\mathbf{S}}(t)$ under the condition:

\begin{equation*}
	\lambda_c \leq \frac{K\ell(r)}{\mathfrak{L}\langle \ell^\varepsilon_\D \rangle}.
\end{equation*}

To obtain the reciprocal, we define $i_C^\dagger(t)$ such that $\underline{Z}_{i_C^\dagger(t),C}(t) = \min_i \underline{Z}_{i,C}(t)$. We can then use the same arguments as developed for Theorem~\ref{thm:reciprocal} to obtain that the system is unstable if $\lambda > \frac{K \ell(r)}{\mathfrak{L}\langle \ell_{\varepsilon,\D} \rangle}$. This way, we get the reciprocal inequality and we can conclude that:

\begin{equation*}
	\lambda_c = \frac{K\ell(r)}{\mathfrak{L}\langle \ell_\D \rangle},
\end{equation*}

\noindent which concludes the proof of Theorem~\ref{thm:stability} in the case $r > 0$.

\section{Poisson heuristic for the critical arrival rate}

In this section, we provide a heuristic the user densities in the stationary regime in the system that will provide with a heuristic for the critical arrival rate in the system as the largest possible solution of a given system of equations.

\subsection{Poisson heuristic for stationary user densities}\label{ssec:poisson}

Let $\mu_C = \frac{1}{\lvert \D \rvert}\E\left[ \Phi_{C,0}(\D) \right]$ denote the user density of the point process $\Phi_{C,0}$ of users of class $C \in \Pk$ in the stationary regime. The first result we can state about the user densities is the following :

\begin{theorem}\label{thm:intensity_symm}
	In the symmetric system, the stationary user densities verify:

	\begin{equation*}
		\forall C, D \in \Pk, \quad \lvert C \rvert = \lvert D \rvert \Rightarrow \mu_C = \mu_D.
	\end{equation*}
\end{theorem}

The result from Theorem~\ref{thm:intensity_symm} is a logical consequence of the symmetry assumption, applied to the stationary regime.

A first heuristic to estimate $\mu_C$ is the Poisson heuristic: we assume that, in the stationary regime, all the point processes $\Phi_{C,0}$ for $C \in \Pk$ are independent Poisson point processes with intensity $\mu_C^f$. Poisson approximations are tied to replica mean-field methods (see~\cite{bacce21} for instance), where we consider multiple realizations of our stochastic process and interactions between users are picked uniformly at random among replicas of the dynamics. We know that (see~\cite{bacce21}) as the number of replicas go to infinity, the processes behave like independent Poisson point processes, which allow to make this approximation. We define the following heuristic:

\begin{heuristic}\label{heur:poisson}
	We define the \emph{Poisson heuristic} $\mu_C^f$ for $\mu_C$ as the smallest solution of the following equation:
	
	\begin{equation}\label{eq:poisson}
	\mu_C^f\int_0^\infty e^{-z\mathcal{N}_0} e^{-\sum_{U \in \Pk}\mu_U^f \mathcal{I}(z,\lvert C \cap U \rvert)}\diff z = \frac{\lambda p_C L_C}{\lvert C \rvert \ell(r)}.
	\end{equation}
	
	where $\mathcal{I}(z,k) = \int_\D(1-e^{-z k \ell(\lVert x \rVert)})\diff x$.
\end{heuristic}

We numerically observe that the values of the intensities $\mu_C^f$ obtained by solving~\eqref{eq:poisson} do not depend on the value of $\mathcal{N}_0 > 0$. This observation is consistent with the fluid system and the stability condition in the network being independent of the value of $\mathcal{N}_0$.

\begin{proof}

	Let us apply Miyazawa's Rate Conservation Principle (cf~\cite{miyaza94}) to the  process $\Phi_{0,C}$ of users of class $C \in \Pk$: during a time interval $dt$, in the stationary regime, a $\lambda p_C L_C \lvert D \rvert dt$ users arrive on average in the network, and in a given network configuration $\Phi_{0,C}$, $\E \left[\frac{1}{L_C}\sum_{x\in \Phi_{C,0}}R(x, \Phi_0) dt\right]$ users leave, which gives:
	
	\begin{equation*}
		\lambda p_C L_C \lvert \D \rvert = \E\left[\sum_{x\in \Phi_{C,0}}R(x, \Phi_0)\right].
	\end{equation*}
	
	We use the definition of the Palm probability measure to rewrite this as:
	
	\begin{equation}\label{eq:RCL}
	\lambda p_C L_C \lvert \D \rvert= \E^0_{\Phi_{C,0}}\left[R(0^C,\Phi_0)\right] \E\left[\Phi_0^C(\D)\right].
	\end{equation}
	
	Using this, we get that, in the stationary regime:
	
	\begin{align*}
	\lambda p_C L_C &= \mu_C \E^0_{\Phi_{C,0}}\left[R(0^C,\Phi_0)\right]\\
	&= \lvert C \rvert \mu_C \ell(r)\mathbb{E}_{\Phi_{C,0}}^0\left[\frac{1}{N_0 + I_C} \right].
	\end{align*}
	
	This is equivalent to:
	
	\begin{equation*}
	\mu_C = \frac{\lambda p_C L_C}{\lvert C \rvert \ell(r)}\left(\mathbb{E}_{\Phi_{C,0}}^0\left[\frac{1}{N_0 + I_C} \right]\right)^{-1}.
	\label{eq:intensity}
	\end{equation*}

	We use the result of Lemma 1 from~\cite{sanka16}:
	
	\begin{lem}\label{lem:expectation}
		Let $Y$ be a positive random variable with finite expectation and $c>0$ be a real number. Then:	
		\begin{equation*}
		\mathbb{E}\left[\frac{1}{c+Y} \right] = \int_0^\infty e^{-zc} \mathbb{E}[e^{-zY}] \diff z.
		\end{equation*}
	\end{lem}
	
	Using Lemma~\ref{lem:expectation} with $Y\equiv I(0^C, \Phi_0)$, which has a finite first moment and with $c \equiv \mathcal{N}_0$, we get:
	
	\begin{multline*}
		\mathbb{E}_{\Phi_{C,0}}^0\left[\frac{1}{\mathcal{N}_0 + I_C} \right] = \int_0^\infty e^{-z\mathcal{N}_0} \mathbb{E}_{\Phi_{C,0}}^0\left[e^{-zI_C}\right] \diff z.
	\end{multline*}
	
	Let us assume that the processes $\Phi_{0,U}$ are independent Poisson point processes denoted by $\Psi^U$. Using this, combined with Slivnyak's theorem, we get:
	

	\begin{multline*}
		\mathbb{E}_{\Phi_{C,0}}^0\left[\frac{1}{N_0 + I_C} \right] \simeq \int_0^\infty e^{-z\mathcal{N}_0} \prod_{U \in \Pk} \mathbb{E}\left[e^{-z\sum_{x\in \Phi_0^{U}\backslash\{0^C\}} \lvert C \cap U \rvert \ell(\lVert x \rVert)}\right] \diff z.
		\end{multline*}
	
	Using the formula for the Laplace transform of a Poisson point process, we get:
	

	\begin{multline*}
		\mathbb{E}\left[e^{-z \sum_{x\in \Phi_0^{U}\backslash\{0^C\}}\lvert C \cap U \rvert \ell(\lVert x \rVert)}\right] =	\exp \left[ - \mu_U^f \int_\D (1-e^{-z\lvert C \cap U \rvert \ell(\lVert x \rVert)})\diff x \right].
	\end{multline*}

	To conclude the proof, we combine these equations:
	
	\begin{equation*}
	\mu_C^f\int_0^\infty e^{-z\mathcal{N}_0} e^{- \sum_{U \in \Pk} \mu_U^f \mathcal{I}(z,\lvert C \cap U \rvert)}\diff z =  \frac{\lambda p_C L_C}{\lvert C \rvert \ell(r)},
	\end{equation*}
	
	\noindent which is the announced result.
\end{proof}

We can use the result from Theorem~\ref{thm:intensity_symm} to obtain another version of the system of equations in~\eqref{eq:poisson}: if we define $\mu_j = \frac{1}{\lvert \D \rvert} \E\left[\sum_{U : \lvert U \rvert = j}\Phi_{U,0}(\D)\right] = \sum_{U: \lvert U \rvert = j} \mu_U$ as the intensity of users communicating on $j$ channels, for $1 \leq j \leq K$, with $L_j = L_C$ and $p_j = \binom{K}{j}p_C$ for all $C$ such that $\lvert C \rvert = j$, we obtain the following system of equations from Heuristic~\ref{heur:estimate}:

\begin{align}\label{eq:poisson_sym}
	&\mu_j^f\int_0^\infty e^{-z\mathcal{N}_0 -\sum_{l = 1}^K \mu_l^f \sum_{m = 1}^{ j\wedge l} \alpha_{m,j,l} \mathcal{I}(z, m)}\diff z =	\frac{\lambda p_j L_j}{j \ell(r)},
\end{align}

\noindent where $\alpha_{m, j, l} = \frac{\binom{j}{m}\binom{j-l}{l-m}}{\binom{K}{j}}$. This result if a consequence of Lemma~\ref{lem:symmetric} applied to the vector of user densities.

This result can be interesting computationally, in the case where we want to quickly estimate the density of users transmitting on a given class in the symmetric system: in contrast to the system proposed in Heuristic~\ref{heur:poisson}, which needs to solve $2^K-1$ equations, this symmetric system possesses $K$ equations. As $K$ grows larger, the gain in performance becomes non-negligible.

\subsection{A heuristic for $\lambda_c$}

An interesting result arising from Heuristic~\ref{heur:poisson} is an estimate for the value of $\lambda_c$: a necessary condition for the fixed point equation~\eqref{eq:poisson} to admit admit solutions is that the system is in it stationary regime, i.e., if $\lambda < \lambda_c$. We define the following estimate for $\lambda_C$:

\begin{heuristic}\label{heur:estimate}
	Let $\lambda_\textrm{P}$ be the largest value of $\lambda$ such that~\eqref{eq:poisson} admits a solution (or~\eqref{eq:poisson_sym} in the symmetric network). Then, $\lambda_\textrm{P}$ is an estimate for $\lambda_c$, which we call the \emph{Poisson heuristic} estimate.
\end{heuristic}

This estimate is interesting because it comes from a completely different approach and is only relying on the study of the stationary regime of Poisson approximates of the network. We will discuss in Section~\ref{sec:numerical} the performance of this heuristic.

The reason why this heuristic captures the stability region of our dynamics is not yet understood. We do not know if the value of $\lambda_\textrm{P}$ is equal to the value of $\lambda_c$ presented in Theorem~\ref{thm:stability} or if this value is a precise numerical approximation of the critical arrival rate.

\subsection{Second order heuristic}\label{ssec:cavity}

In an effort to refine the first order heuristic for user densities from Heuristic~\ref{heur:poisson}, we introduce a \emph{second order} heuristic, which uses an approximation for the pair-wise correlation function in the stationary regime.

\begin{heuristic}\label{heur:second_order}
	The \emph{second-order heuristic} for the intensity of $\Phi_{0,C}$ is $\mu^s_C$ defined as:
	\begin{equation*}
		\mu^s_C = \frac{\lambda p_C L_C}{\lvert C \rvert \ell(R)}\left(\mathcal{N}_0 + I_C\right),
	\end{equation*}

	\noindent where $I_C \triangleq \E^0_{\Phi_{C,0}}\left[I(0^C, \Phi_0)\right]$ is the interference experienced by the typical user of class $C$ in the stationary regime. The vector $\mathbf{I} = (I_C)_{C\in\Pk}$ is the solution to the following equation:

	\begin{equation*}
		I_C = \sum_{U \in \Pk} \frac{\lvert C \cap  U \rvert}{\mu^s_C}\int_{x \in \D} \ell(\lVert x \rVert) \rho^{(2)}_{C,U}(x, 0) \diff x,
	\end{equation*}

	\noindent with the second order moment measure $\rho^{(2)}_{C,U}$ is defined as a function of $\mathbf{I}$: 
	
	\begin{align}\label{eq:inter_fp}
		&\rho^{(2)}_{C,D}(x,y) = \frac{\lambda(\mu^s_D p_C + \mu^s_C p_D)}{d_{C,D}(x, y)},
	\end{align}

	\noindent and: 

	\begin{align*}
		d_{C,D}(x, y) = &\frac{1}{L_C}\frac{\lvert C\rvert \ell(R)}{\mathcal{N}_0 + \lvert C \cap D \rvert \ell\left(\lVert x - y \rVert\right) + I_C} \\
		&+ \frac{1}{L_D}\frac{\lvert D\rvert \ell(R)}{\mathcal{N}_0 + \lvert C \cap D \rvert \ell\left(\lVert x - y \rVert\right) + I_D}.
	\end{align*}

\end{heuristic}

\begin{proof}
	To obtain this heuristic, we make the following cavity approximation in the system: let us consider a pair of points $(x,y)$ in $\D$, where $x$ is of class $C$ and $y$ is of class $D$.

	The arrival rate of the pair is $\lambda p_C \mu_D + \lambda p_D \mu_C$, taking into account the contributions of the two processes. To obtain the departure rate, let us fix the locations of the two points. The interference experienced by a point located at $x$ of class $C$ in the stationary regime conditioned on a user of class $D$ being present at $y$ is equal to:

	\begin{equation}\label{eq:stat_inter}
		I(x) = \lvert C \cap D \rvert \ell(\lVert x - y \rVert) + I_C,
	\end{equation}

	\noindent where $I_C$ is the interference experienced by the typical user of class $C$ in the network, i.e. $I_C = \E^0_{\Phi_0}\left[I(0^C, \Phi_0)\right]$. Using the definition of the second order moment measure relative to processes $\Phi_{C,0}$ and $\Phi_{U,0}$, $\rho^{(2)}_{C,U}$, we have:

	\begin{align*}
		I_C &= \E^0_{\Phi_0}\left[I(0^C, \Phi_0)\right] \\
		&= \sum_U \E^0_{\Phi_0}\left[I(0^C, \Phi_0^U)\right] \\
		&= \frac{1}{\mu_C} \sum_U \lvert C \cap U \rvert \int_{x \in \D} \ell(\lVert x \rVert) \rho^{(2)}_{C,U}(x,0) \diff x.
	\end{align*}
	
	We obtain the departure rate for the pair $(x,y)$ with $x$ being of type $C$ and $y$ being of type $D$:

	\begin{align*}
		d_{C,D}(x, y) = &\frac{1}{L_C}\frac{\lvert C\rvert \ell(R)}{\mathcal{N}_0 + \lvert C \cap D \rvert \ell\left(\lVert x - y \rVert\right) + I_C} \\
		&+ \frac{1}{L_D}\frac{\lvert D\rvert \ell(R)}{\mathcal{N}_0 + \lvert C \cap D \rvert \ell\left(\lVert x - y \rVert\right) + I_D}.
	\end{align*}

	In the stationary regime, on average, the umber of arrivals compensate the departures, which gives us:

	\begin{equation*}
		\lambda(\mu_D p_C + \mu_C p_D) = \rho^{(2)}_{C,D}(x,y) d_{C,D}(x, y).
	\end{equation*}

	Finally, to obtain the second-order heuristic for $\mu_C$, we use Miyazawa's Rate Conservation principle for the process $\Phi_{C,0}(\D)$, assuming that all the quantities are distributed according to their stationary distribution:

	\begin{equation*}
		\lambda p_C \lvert \D \rvert = \E\left[\Phi_{C,0}(\D)\right] \frac{\lvert C \rvert \ell(r)}{\mathcal{N}_0 + I_C},
	\end{equation*}

	\noindent which gives the intended result.

\end{proof}

\section{Numerical simulations and performance of the heuristics}\label{sec:numerical}

\subsection{Stability region}

In this section, we discuss the performance of the two heuristics developed in Heuristics~\ref{heur:poisson} and~\ref{heur:second_order}. In~\ref{fig:heuristic}, a symmetric setup with 2 channels and three classes denoted as $\{1\}$, $\{2\}$ and $\{1,2\}$ is considered. The average file sizes  are $L_{\{1\}} = L_{\{2\}} = 1$ and $L_{\{1,2\}} = 2$ and the probability distribution is $p_{\{1\}} = p_{\{2\}} = 0.4$ and and $p_{\{1,2\}} = 0.2$. The path-loss function is $\ell(x) = (1+x)^{-4}$, and the domain $\D$ is the square torus centered at the origin with side length 10.

To simulate the dynamics of the system, we simulate the embedded Markov chain in the process $\Phi_t$: at each iteration, we compute the time to the next birth by drawing a time $t_b$ from an exponential distribution with mean $\lambda \lvert \D \rvert$ and a random vector $\mathbf{t}_D$ where each coordinate corresponds to a user located at $x_i$ of class $C_i$ and is drawn from an exponential distribution with mean $\frac{1}{L_{C_i}} R(x_i, \Phi_t)$. 

We then compare the values of $t_b$ and the minimum of $\mathbf{t}_D$: if the former is lower, a user arrives at a location $x$ taken uniformly at random in $\D$; if the latter is lower, the corresponding user leaves the system. We then update the value of the interference in the system and move forward in time to the next event. If $(T_0 = 0, T_1, \dots, T_n)$ are the event times in the network and $X$ is the embedded Markov chain, we obtain that $\Phi_{T_n} = X_n$, which gives us the representation of process $\Phi_t$. Figure~\ref{fig:pop} shows the behavior of the system with $K=2$ for two values of $\lambda$: one that shows stability, and the other, instability.

\begin{figure}[h]
	\centering
	\begin{minipage}[b]{0.45\textwidth}
		\centering
		\includegraphics[width=\linewidth]{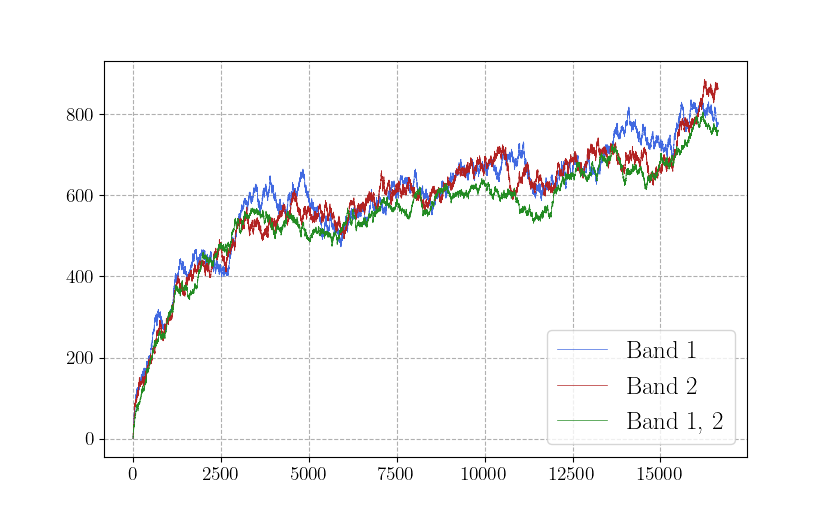}
	\end{minipage}
	\hfill
	\begin{minipage}[b]{.45\textwidth}
		\centering
		\includegraphics[width=\linewidth]{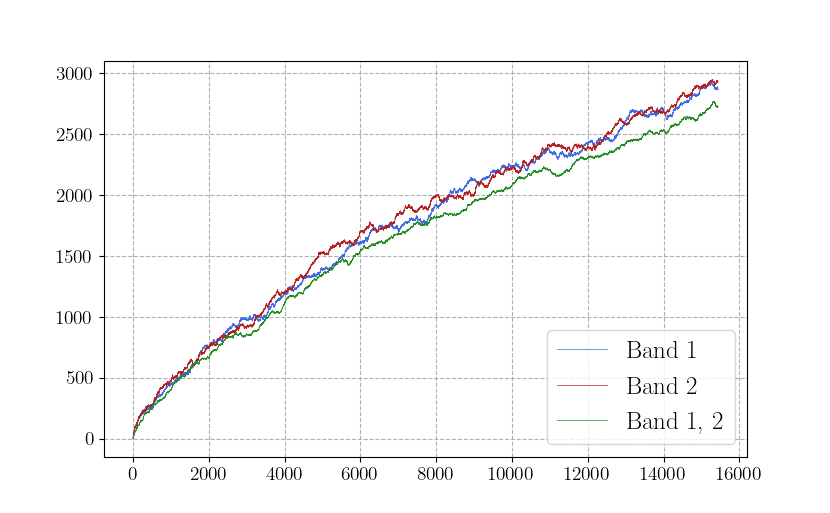}
	\end{minipage}
	\caption{Number of active users in the network over time with $K = 2$ bands for two values of $\lambda$. On the left, $\lambda = 0.95 \lambda_c$ and we can see that the number of users in the network stays bounded. On the right, $\lambda = 1.05 \lambda_c$. In this case, population grows linearly over time in the network, showing the instability of the system.}
	\label{fig:pop}
\end{figure}

A reliable criterion to study the stability of the system is Little's law (see~\cite{little61}), that links the average staying time $W$ of a user in the system, the arrival rate $\lambda$ and the number of users $L$ in the system in the stationary regime:

\begin{equation}\label{eq:little}
	L = \lambda W.
\end{equation}

Using the parameters for our system, \eqref{eq:little} becomes:

\begin{equation}\label{eq:little_2}
	\E \left[\Phi_{t,C} (\D)\right] = \lambda p_C W_C,
\end{equation}

\noindent where $W_C$ is the staying time of a user of class $C$ and $\E \left[\Phi_{t,C} (\D)\right]$ is the average number of users of class $C$ in the stationary regime. When simulating the dynamics of the network, we can compute the average time spent by users in the network in each class. If the value converges to a fixed, finite value, then the network is stable and the number of users follow~\eqref{eq:little_2}. If the time lived in the network grows linearly and diverges to infinity, then the network is not stable.

Figure~\ref{fig:delay} shows the average staying time of users in the network for two different values of $\lambda$. The phase transition in the network happens as expected: for $\lambda = 0.9\lambda_c$, the system appears to be stable and for $\lambda = 1.1 \lambda_c$, the average staying time grows linearly, meaning that no stationary regime exists for in the network, and that our dynamics are unstable.

\begin{figure}[h]
	\centering
	\begin{minipage}[b]{0.45\textwidth}
		\centering
		\includegraphics[width=\linewidth]{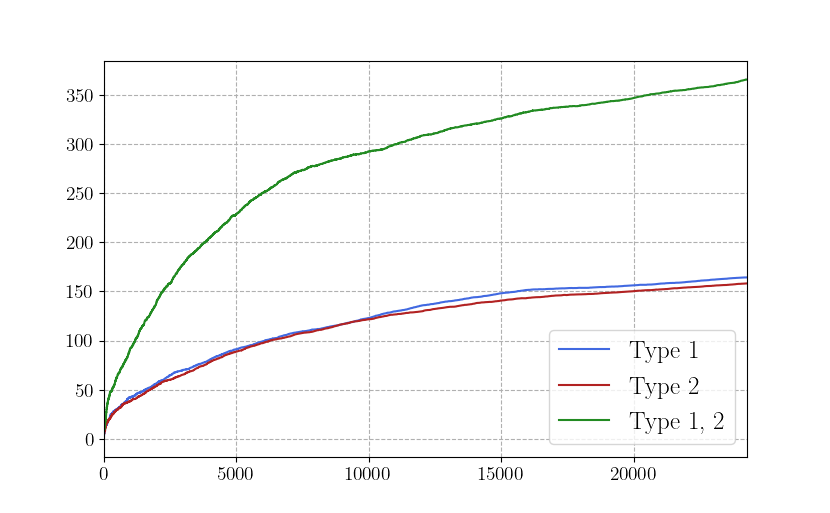}
	\end{minipage}
	\hfill
	\begin{minipage}[b]{.45\textwidth}
		\centering
		\includegraphics[width=\linewidth]{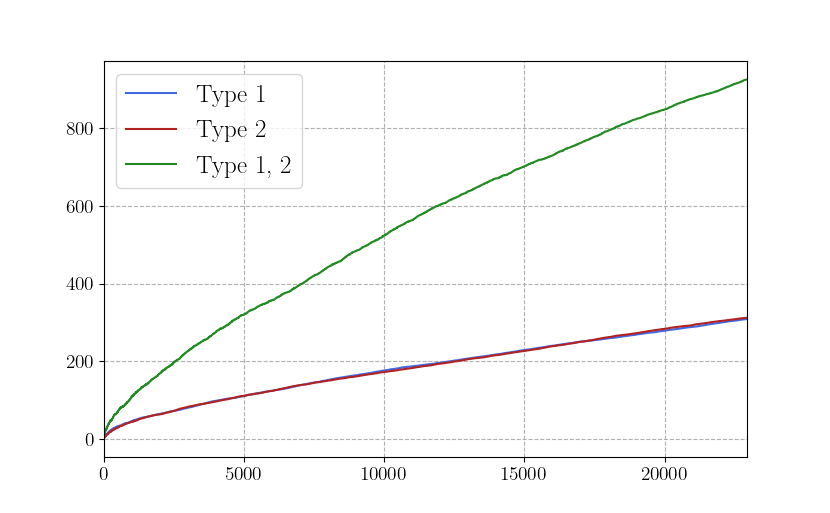}
	\end{minipage}
	\caption{Average staying time of users in the network with $K=2$ channels for two different values of $\lambda$. In blue and red, users transmitting only on one channel, in green, users transmitting on the two channels. On the left, $\lambda = 0.9\lambda_c$, on the right $\lambda = 1.1\lambda_c$. The left network is stable, and the average staying time converges to a finite limit. The right network is unstable and the staying time diverges.}
	\label{fig:delay}
\end{figure}

\subsection{Poisson heuristics}

We can compare the heuristics from Heuristic~\ref{heur:poisson} and~\ref{heur:estimate} to the value we obtain through simulation. To compute this intensity, we use an ergodic approximation of the form:

\begin{equation}\label{eq:ergodic_mean}
	\mu_C \approx \frac{1.28}{t} \int_{u}^{u+t} \frac{\Phi_{w,C}(\D)}{\lvert \D \rvert} \diff w,
\end{equation}

\noindent where $u$ is taken sufficiently large to ensure we are in the stationary regime of the system and $t$ is large enough so that the approximation is correct. The 1.28 factor arising in~\eqref{eq:ergodic_mean} comes from a Palm bias: in the simulation, when we compute the average number of users in $\D$, we introduce an observation bias when estimating the stationary user density, which we correct by using this multiplicative factor.

Figure~\ref{fig:heuristic} displays the results of both the Poisson (in red, dashed) against the numerical estimation of the user density (in blue, plain) for the process $\Phi_{\{1\}}$ using~\eqref{eq:ergodic_mean}.

\begin{figure}[h]
	\centering
	\includegraphics[width = \linewidth]{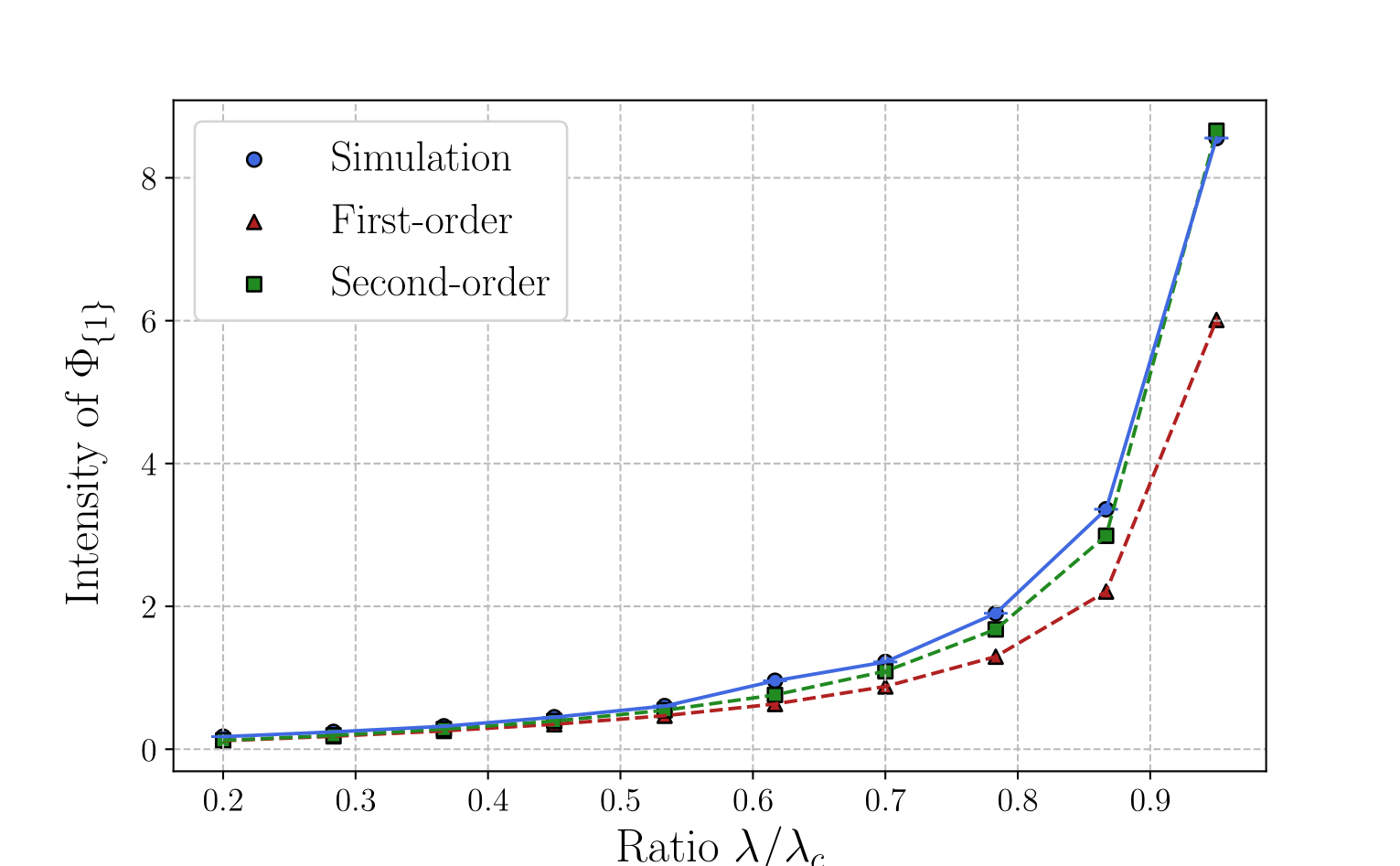}
	\caption{Value of $\mu_{\{1\}}$ as a function of $\lambda/\lambda_c$ in a symmetric configuration.}
	\label{fig:heuristic}
\end{figure}

As we expected from previous results (see Section IV from~\cite{sanka16}), the Poisson heuristic for spatial wireless networks does not estimate precisely the stationary user densities. This result can be explained by the existence of \emph{clustering}, or attraction, in the stationary system. Clustering implies that the spatial positions of users in the stationary regime are not independent, and that connecting users tend to attract other users by slowing communications of users that stay too close to them. A sensible conjecture to make is that our spatial multiclass wireless network displays the same behavior in the stationary regime, which would provide an explanation for this observation.

The Poisson estimate gives us an interesting result: Figure~\ref{fig:l_star} shows the respective values of $\lambda_c$ and $\lambda_\textrm{P}$ in the symmetric case for different values of $p_{\{1,2\}}$, with $p_{\{1\}} = p_{\{2\}} = \frac{1-p_{\{1,2\}}}{2}$. The two values are very close, which means that the Poisson hypothesis we make in the stationary regime is a reasonable approximation to obtain a condition for the stability of the system. In other words, Heuristic~\ref{heur:estimate} provides an accurate numerical heuristic for the result of Theorem~\ref{thm:stability}.

\begin{figure}[h]
	\centering
	\includegraphics[width = \linewidth]{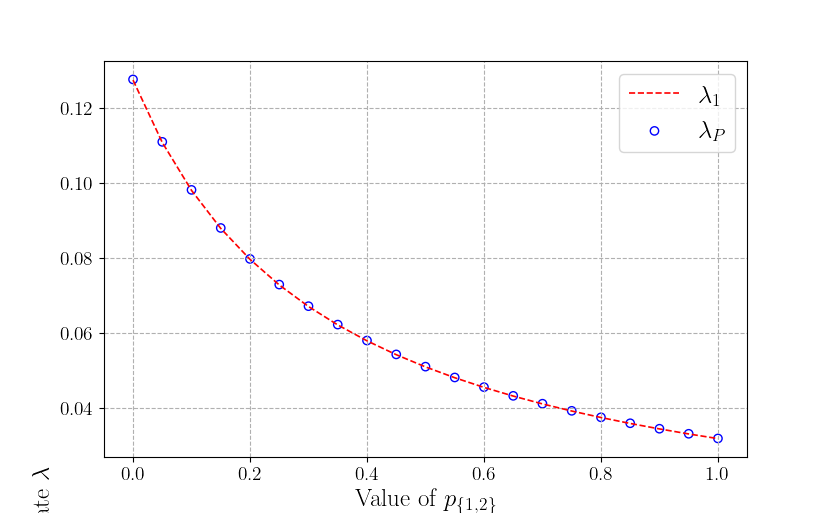}
	\caption{Comparison between $\lambda_c$ and its Poisson heuristic as a function of $p_{\{1,2\}}$ in the network in the symmetric case.}
	\label{fig:l_star}
\end{figure}

\section{Discussion and extensions}

This paper proposed a first step in the analysis of a multiclass interacting point process dynamics
representing important new features of wireless networks, and which extends the single-class model of~\cite{sanka16}
by adding service differentiation. Service differentiation changes quite significantly the interactions between
users, and new ideas were needed to assess the monotonicity, the $\phi$-irreducibility,
the stochastic stability, and the steady state properties for the associated dynamics.

Here are, to conclude, a few interpretation comments and natural extensions of this first step.

\subsection{Interpretation of the results}

Theorem~\ref{thm:stability} gives the critical arrival rate below which the dynamics is stable and above which it is not. 
The form of the critical value in question is somewhat unexpected. The quantity
$\mathfrak{L} = \sum_U \lvert U \rvert p_U L_U$ which shows up in this critical value can be given
the following interpretation: on average, a user has $L_U$ bits to transmit over $\lvert U \rvert$ bands.
Hence, a user of type $U$ brings a time-space {\em load}  equal to $\lvert U \rvert L_U$ to the network.
Taking the average over the distribution of user types, we obtain $\mathfrak{L}$.
We can rewrite the stability condition $\lambda < \frac{K\ell(R)}{\langle \ell_\D \rangle \mathfrak{L}}$ as:

\begin{equation*}
    \lambda \frac{\mathfrak{L}}{K} < \frac{\ell(R)}{\langle \ell_\D \rangle},
\end{equation*}

We can interpret the left hand side term as the average {\em load} arriving per unit of time and per channel.
The right-hand side is the throughput capacity of one channel in heavy load
(the signal power divided by the fluid interference power).

\subsection{Possible extensions}

In this paper, we made the assumption that the system was in a low-SINR regime, 
which justified the linearization of the Shannon formula in~\eqref{eq:rof_lin}. 
This is not an essential limitation and it is not difficult to show that the stability results
we present here remain valid, up to a multiplicative factor, when using the non-linearized 
version of the Shannon formula.

The methodology can also be extended and applied to more complex network settings.

A first natural extension concerns the relaxation of the symmetry hypothesis we made.
Relaxing this assumption prevents the use of Lemma~\ref{lem:symmetric} and the stability condition
of Theorem~\ref{thm:stability} does not work anymore. For a non-symmetric multiclass network,
it is not clear whether we will get a closed-form for the critical arrival rate.

A second question comes from the Poisson heuristic we defined in Heuristic~\ref{heur:estimate}.
This heuristic gives a precise estimate for the value of $\lambda_c$ in the symmetric system.
So far, we cannot say whether this estimate is a precise numerical approximation or an 
exact result for $\lambda_c$. A first step to answer this question would be to study the
spatial correlation of points when reaching the limit of the stability region:
if this spatial correlation vanishes, the Poisson approximation becomes relevant and
this could explain why this heuristic performs well. A result supporting this 
interpretation comes from Section V of~\cite{sanka16}, where the authors showed that
clustering becomes less prominent when $\lambda$ becomes closer to $\lambda_c$. 
The question whether this extends to the multiclass setting is open.
It is also worth noting that in simulations not presented in this paper,
the Poisson heuristic close to $\lambda_c$ still performs well when removing
the symmetry hypothesis, and in different numerical setups as the one presented in Section~\ref{sec:numerical}.

Among the other interesting properties we can expect to extend from single to multiclass,
let us quote the existence and quantification of clustering (see Definition 1 of~\cite{sanka16}),
the definition of more precise stationary user density heuristics and the study of higher order moment measures.
We already know that, for single class interference queueing networks, the system has exponential
moments in the stationary regime (see~\cite{sanka19}). Extending the result to our multiclass
dynamics willl help quantifying the tail distribution of the service time, which is critical for 5G networks.

Lastly, another question would be the extension of our results to the whole Euclidean plane. This will however require 
completely different proof techniques as there is no theory for assessing the stability of infinite dimensional
systems using fluid limits.

\section*{Acknowledgements}
This work was supported by the ERC NEMO grant, under the European Union's Horizon 2020 research
and innovation programme, grant agreement number 788851 to INRIA.
The authors thank S. Foss and S. Shneer of Heriot Watt University for their comments and suggestions on this work.

\bibliographystyle{ieeetr}
\bibliography{BWP}


\section{Appendix}

\section*{A: Proof of Theorem~\ref{thm:FSP}}\label{sec:AppendixA}

Let us take a sequence of initial conditions $\bar{\mathbf{X}}_{\varepsilon}^n(0) = (\bar{x}^n_{i,C})$ for our system, with $\lim_{n \rightarrow \infty} \bar{x}^n_{i,C} = \infty$. The goal in this section is to study the limit of the sequence of fluid-scaled processes $\left(\frac{1}{n}\bar{\mathbf{X}}_\varepsilon(nt)\right)$ and to obtain a system of ODEs for which the limit of this sequence is a solution.

Using~\eqref{eq:dynamic}, we get, for all $i, C$:


\begin{align}\label{eq:scaled}
	&\frac{1}{n}\bar{X}_{i,C}(nt) = \frac{\bar{x}^n_{i,C}}{n} + \frac{1}{n}\mathcal{A}_{i,C}(\lambda p_C \varepsilon^2 n t)
	- \frac{1}{n}N_{i,C}\left( \frac{1}{L_C}\int_0^{nt} \bar{X}_{i,C}(u)R_{i,C}(\bar{\mathbf{X}}_\varepsilon(u)) \diff u\right).
\end{align}

We use the variable change $u = ns$ in the integral to get:


\small
\begin{align*}
	&\int_0^{nt} \bar{X}_{i,C}(u)R_{i,C}(\bar{\mathbf{X}}_\varepsilon(u)) \diff u = \int_0^{t} \frac{n\lvert C \rvert \frac{1}{n}\bar{X}_{i,C}(ns)\diff s}{\frac{\mathcal{N}_0}{n} + \sum_{k,U}\ell^\varepsilon(a_k, a_i) \lvert C \cap U \rvert \left(\frac{1}{n}\bar{X}_{k,U}(ns) - \frac{1}{n}\mathbf{1}_{\{U = C, i = k\}}\right)}.
\end{align*}
\normalsize

We define:


\begin{multline*}
	R_{i,C}^n(x) = \frac{\lvert C \rvert}{\frac{\mathcal{N}_0}{n} + \sum_{k,U}\ell^\varepsilon(a_k, a_i) \lvert C \cap U \rvert \left(x_{k,U} - \frac{1}{n}\mathbf{1}_{\{U = C, i = k\}}\right)}.
\end{multline*}

Using this, \eqref{eq:scaled} becomes:

\begin{multline*}
	\frac{1}{n}\bar{X}_{i,C}(nt) = \frac{1}{n}\bar{x}^n_{i,C} + \frac{1}{n}\mathcal{A}_{i,C}(\lambda p_C \varepsilon^2 nt) \hfill \\
	- \frac{1}{n} N_{i,C}\left( \frac{n}{L_C}\int_0^t \frac{1}{n}\bar{X}_{i,C}(ns) R_{i,C}^n\left(\frac{1}{n} \bar{\mathbf{X}}_{\varepsilon}(ns)\right)\diff s\right),
\end{multline*}

\noindent Let $M^n_{i,C}(z) = \frac{1}{n}N_{i,C}(nz) - z$ for $z \in \R$, and let us define:


\begin{multline*}
	\bar{Y}^n_{i,C}(t) = \frac{1}{n}\bar{x}^n_{i,C} + \frac{1}{n}\mathcal{A}_{i,C}(\lambda p_C \varepsilon^2 nt) - M^n_{i,C}\left( \frac{1}{L_C}\int_0^t \frac{1}{n}\bar{X}_{i,C}(ns)R_{i,C}^n\left(\frac{1}{n}\bar{\mathbf{X}}_{\varepsilon}(ns)\right) \diff s\right).
\end{multline*}

From~\eqref{eq:scaled}, we have:

\begin{multline*}
	\frac{1}{n}\bar{X}_{i,C}(nt) = \bar{Y}^n_{i,C}(t) - \frac{1}{L_C}\int_{0}^t \frac{1}{n}\bar{X}_{i,C}(ns)R_{i,C}^n\left(\frac{1}{n}\bar{\mathbf{X}}_\varepsilon(ns)\right)\diff s.
\end{multline*}

Let us denote $\bar{\mathbf{x}}_0 = \lim_{n \rightarrow \infty} \frac{1}{n}\bar{x}^n_{i,C}$ when it exists. We say that $\bar{\mathbf{x}}_0$ is finite if and only if all its coordinates are finite.

\begin{lem}\label{lem:convergence}
	If $\bar{\mathbf{x}}_0$ exists and is finite, then:
	\begin{equation}
		\forall t \geq 0, \forall i, C, \quad \lim_{n \rightarrow \infty}\bar{Y}^n_{i,C}(t) = \bar{X}_{i,C}(0) + \lambda p_C \varepsilon^2 t. \label{eq:cv_y}
	\end{equation}
\end{lem}

\begin{proof}
	Using the strong law of large numbers yields, $\Proba$-almost surely:
	
	\begin{equation*}
		\lim_{n \rightarrow \infty}\frac{1}{n}\mathcal{A}_{i,C}(\lambda p_C \varepsilon^2 nt) = \lambda p_C \varepsilon^2 t.
	\end{equation*}
	
	Moreover, we have:
	
	\begin{align*}
		& \frac{1}{L_C}\int_{0}^t\frac{1}{n} \bar{X}_{i,C}(ns) R_{i,C}^n\left(\frac{1}{n} \bar{\mathbf{X}}_{\varepsilon}(ns)\right) \diff s\\
		& = \frac{1}{L_C}\int_{0}^t \frac{\lvert C \rvert \frac{1}{n} \bar{X}_{i,C}(ns)}{\frac{N_0}{n} + \sum_{k,U} \lvert C \cap U \rvert \ell^\varepsilon(a_k, a_i) \frac{1}{n}\bar{X}_{k,U}(ns)} \diff s \\
		& \leq \frac{1}{L_C} \int_{0}^t \frac{\lvert C \rvert \frac{1}{n} \bar{X}_{i,C}(ns)}{\lvert C \rvert \ell^\varepsilon(a_i, a_i) \frac{1}{n}\bar{X}_{i,C}(ns)} \diff s \\
		& = \frac{t}{L_C},
	\end{align*} 
	
	\noindent where we use the fact that for all $0 \leq i \leq N_\varepsilon - 1$, $\ell_\varepsilon(a_i, a_i) = 1$.

	We use Prohorov's Theorem (see Theorems 5.1 and 5.2 of~\cite{billi99}) to obtain that the sequence of processes $\left\{\frac{1}{n}\bar{X}_{i,C}(n\cdot) \right.$ $\left.R^n_{i,C}(\frac{1}{n}\bar{\mathbf{X}}_\varepsilon(n\cdot))\right\}$ is tight. Finally, because $M^n_{i,C}(z) \rightarrow 0$ as $n$ goes to infinity for all $z$ (using the strong law of large numbers), we can conclude that:
	
	\begin{equation*}
		\lim_{n \rightarrow \infty} M^n_{i,C}\left( \frac{1}{L_C}\int_{0}^t \frac{1}{n} \bar{X}_{i,C}(ns) R_{i,C}^n\left(\frac{1}{n} \bar{\mathbf{X}}_\varepsilon(ns)\right)\diff s\right) = 0,
	\end{equation*}
	
	\noindent which leads to the intended result.
\end{proof}

For $v : [0,T) \rightarrow \R$ and $0 \leq t_1 \leq t_2 \leq T$, we define:

\begin{equation*}	
	w(v, [t_1, t_2]) = \sup \left\{ \lvert v(u_1)-v(u_2) \rvert, u_1, u_2 \in [t_1, t_2] \right\}. 
\end{equation*}

Let us define the modulus of continuity $\omega(v, \delta, T) = \sup \{w(v,[s, t]), 0 \leq s,t \leq T, \lvert s-t \rvert < \delta\}$. We introduce the concept of \emph{C-tightness} and a useful characterization (see Definition VI.3.25 and Proposition VI.3.26 from~\cite{jacod87}):

\begin{definition}[C-tightness]
	A sequence $\{V^n, n\geq 1\}$ of functions is \emph{C-tight} if and only if, for all $t_0> 0$, $\eta >0$, $T > 0$, there exists $K^0_\eta$, $n^0_\eta$ and $\delta^0_\eta$ such that, for all $n \geq n^0_\eta$:
	
	\begin{itemize}
		\item[i)]  $\Proba\left[ \sup_{0\leq t \leq T}\lvert V^n(t)\rvert \geq K^0_\eta \right] < \eta$.
		\item[ii)] $\Proba\left[\omega(V^n, \delta^0_\eta, t_0) > \eta \right] < \eta$.
	\end{itemize}
\end{definition}

Condition $ii)$ implies that any limit point of $V^n$ has continuous sample paths, $\Proba$-almost surely: let $V$ be a limit point of the sequence $V^n$. By definition, $\Proba\left[\omega(V, \delta^0_\eta, t_0) > \eta \right] < \eta$ for all $\eta > 0$, i.e., $\Proba$-almost surely, $\omega(V, \delta^0, t_0) > \eta$ for each $T>0$ and $\delta_C>0$. By continuity of the function $\delta \mapsto \omega(V, \delta, t_0)$ at 0, we can conclude that $V$ is $\Proba$-almost surely continuous.

We can prove the following lemma:

\begin{lem}\label{lem:c-tight}
	If $\bar{\mathbf{x}}_0$ exists and is finite, the sequence of processes $\{\frac{1}{n}\bar{\mathbf{X}}_{\varepsilon}(n \cdot), n\geq 0\}$ is C-tight.
\end{lem}

\begin{proof}
	Let $T>0$ and $\eta > 0$. As $n$ goes to infinity, we know that $\bar{Y}^n_{i,C}(t) \rightarrow \bar{X}_{i,C}(0) + \lambda p_C\varepsilon^2t$ $\Proba$-almost surely from Lemma~\ref{lem:convergence}. For each $\eta > 0$, there exists $n^{T,1}_\eta$ such that for $n \geq n^{T,1}_\eta$:
	
	\begin{equation*}		
		\Proba\left[ \sup_{0 \leq t \leq T}\lvert \bar{Y}^n_{i,C}(t)  \rvert \geq \lambda p_C\varepsilon^2 T + 1 \right] \leq \eta.
	\end{equation*}
	
	We also know that $\frac{1}{n}\bar{X}_{i,C}(nt) R_{i,C}^n(\frac{1}{n}\bar{\mathbf{X}}_{\varepsilon}(nt)) \leq 1$ at all times $t$, a fortiori for all $t \leq T$. We deduce that for all $n \geq n^{T,1}_\eta$:

	\begin{equation*}
		\Proba\left[ \sup_{0 \leq t \leq T}\left\lvert \frac{1}{n} \bar{X}_{i,C}(nt)\right\rvert \geq \lambda p_C\varepsilon^2 T + 1 + 1 \right] \leq \eta.
	\end{equation*}

	We set $K^0_\eta = \lambda p_C\varepsilon^2 T + 2$ to obtain condition $i)$.

	To obtain condition $ii)$, we remark that proving the continuity of the limits is equivalent to showing that there exists a $\delta^0_\eta > 0$ such that for all $i, C$ and at all times $t < T$, we have:
	
	\begin{equation*}
		\Proba\left[\omega(\frac{1}{n}\bar{X}_{i,C}(nt), \delta^0_\eta, t_0) > \eta \right] < \eta.
	\end{equation*}
	
	Let $\delta > 0$. We have at all times $t\geq 0$:
	
	\begin{multline*}
		\frac{1}{n} \bar{X}_{i,C}(n(t + \delta))- \frac{1}{n}\bar{X}_{i,C}(nt) = \bar{Y}^n_{i,C}(t + \delta) - \bar{Y}^n_{i,C}(t) \hfill \\
		- \frac{1}{L_C}\int_{t}^{t+\delta}\frac{1}{n}\bar{X}_{i,C}(nu)R_{i,C}^n\left(\frac{1}{n}\bar{\mathbf{X}}_{\varepsilon}(nu)\right) \diff u.
	\end{multline*}
	
	Taking the supremum over $t \in [0,T]$, and reminding that $\frac{1}{n}\bar{X}_{i,C}(nt) R_{i,C}^n(\frac{1}{n}\bar{\mathbf{X}}_{\varepsilon}(nt)) \leq 1$ yields:
	
	\begin{multline*}
		\sup_{0 \leq t \leq T} \left\lvert \frac{1}{n} \bar{X}_{i,C}(n(t + \delta))- \frac{1}{n}\bar{X}_{i,C}(t) \right\rvert \leq \sup_{0 \leq t \leq T} \left\lvert \bar{Y}^n_{i,C}(t + \delta) - \bar{Y}^n_{i,C}(t) \right\rvert + \frac{\delta}{L_C}.
	\end{multline*}
	
	From Lemma~\ref{lem:convergence}, we know that $\bar{Y}^n_{i,C}(t) \rightarrow \bar{X}_{i,C}(0) + \lambda p_C \varepsilon^2 t$ $\Proba$-almost surely. This implies that there exists $n^{T,2}_\eta > 0$ and $\kappa_\eta^0 > 0$ such that, with probability at least $1-\eta$, $\lvert \bar{Y}^n_{i,C}(t + \delta) - \bar{Y}^n_{i,C}(t) \rvert < \lambda p_C \varepsilon^2 \delta + \delta\kappa_\eta^0$.
	
	This implies that for all $i, C$, we have, with probability at least $1-\eta$:

	\begin{equation*}
		\omega\left(\frac{1}{n}\bar{X}_{i,C}, \delta, t_0\right) < \delta \left(\lambda p_C \varepsilon^2 + \kappa_\eta^0 + \frac{1}{L_C}\right).
	\end{equation*}

	Let us set $\delta^0_\eta = \eta \left(\lambda p_C \varepsilon^2 + \kappa_\eta^0 + \frac{1}{L_C}\right)^{-1}$. Thus, we get, for all $n \geq n^{T,2}_\eta$:

	\begin{equation*}
		\Proba\left[ \omega\left(\frac{1}{n}\bar{X}_{i,C}, \delta^0_\eta, t_0\right) < \eta \right] > 1 - \eta.
	\end{equation*}

	Setting $n^0_\eta = \max(n^{1,0}_\eta,n^{2,0}_\eta)$ concludes the proof of the C-tightness of the sequence of processes $(\frac{1}{n}\bar{\mathbf{X}}_\varepsilon(t))$.

\end{proof}

We can now establish the equation ruling the evolution of the fluid scaled model:

\begin{theorem}\label{thm:fluid_equation}
	If $\frac{1}{n}\bar{\mathbf{X}}_{\varepsilon}(0) \rightarrow \bar{\mathbf{x}}_0$ as $n$ goes to $\infty$, then the sequence of processes $\frac{1}{n}\bar{\mathbf{X}}_{\varepsilon}(n\cdot)$ converges $\Proba$-almost surely to $\bar{\mathbf{x}}_\varepsilon(s) = (\bar{x}_{i,C}(s))$, which is the unique solution of the following system of differential equations:

	\begin{equation*}
		\begin{cases}
			\bar{x}'_{i,C}(t) = \lambda p_C \varepsilon^2 - \frac{1}{L_C}\frac{\lvert C\rvert \bar{x}_{i,C}(t)}{\sum_{k, U} \lvert C \cap U \rvert \ell^\varepsilon(a_k, a_i) \bar{x}_{k,U}} \hfill \textrm{ if } \bar{\mathbf{x}}(t) \neq 0 \\
			\bar{\mathbf{x}}_0 = \bar{\mathbf{X}}_\varepsilon(0)
		\end{cases}.
	\end{equation*}
\end{theorem}

\begin{proof}

	From Lemma~\ref{lem:c-tight}, the sequences $\{\frac{1}{n}\bar{\mathbf{X}}_{\varepsilon}(n\cdot), n\geq 1\}$ and $\{\bar{\mathbf{Y}}^n, n\geq 1\}$ are both tight. It follows from Theorem 11.6.8 from~\cite{whitt02} that the sequence $\{(\frac{1}{n}\bar{\mathbf{X}}_{\varepsilon}(n\cdot), \bar{\mathbf{Y}}^n), n\geq 1\}$ is tight in $D([0,\infty], \R^{N_\varepsilon \times 2^K-1})$, and thus, by using Prohorov's Theorem, relatively compact. Let ${n_l}$ be a subsequence along which $\left(\frac{1}{n_l}\bar{\mathbf{X}}_{\varepsilon}(n_l \cdot), \bar{\mathbf{Y}}^{n_l}\right)$ converges to a limit point $(\bar{\mathbf{x}}_{\varepsilon}, \bar{\mathbf{Y}})$ as $l$ goes to infinity.

	We use the Skorokhod representation theorem (see Theorem 6.7 in~\cite{billi99}) to get a probability space $(\hat{\Omega}, \hat{\mathcal{F}}, \hat{\Proba})$ with a sequence of processes $\{(\hat{\mathbf{X}}^{n_l}, \hat{\mathbf{Y}}^{n_l}), l\geq 1\}$ and two processes $\hat{\mathbf{x}}_\varepsilon$ and $\hat{\mathbf{Y}}$ such that:
	\begin{itemize}
		\item $\left(\hat{\mathbf{X}}^{n_l}, \hat{\mathbf{Y}}^{n_l}\right) \rightarrow (\hat{\mathbf{x}}_\varepsilon, \hat{\mathbf{Y}})$, $\hat{\Proba}$-almost surely;
		\item $\left(\hat{\mathbf{X}}^{n_l}, \hat{\mathbf{Y}}^{n_l}\right) \overset{d}{\sim} (\frac{1}{n_l}\bar{\mathbf{X}}_\varepsilon(n_l, \cdot), \bar{\mathbf{Y}}^{n_l})$ for all $l \geq 1$;
		\item $(\hat{\mathbf{x}}_\varepsilon, \hat{\mathbf{Y}}) \overset{d}{\sim} (\bar{\mathbf{X}}_\varepsilon, \bar{\mathbf{Y}})$.
	\end{itemize}

	The second point gives, for all $l \geq 1$:

	\begin{equation}
		\hat{X}^{n_l}_{i,C}(t) = \hat{Y}_{i,C}^{n_l}(t) - \frac{1}{L_C}\int_0^t \hat{X}^{n_l}_{i,C}(s)R_{i,C}^n(\hat{\mathbf{X}}^{n_l}(s))\diff s. \label{eq:hat}
	\end{equation}

	Using the C-tightness of $\bar{\mathbf{X}}_\varepsilon$, $\hat{\mathbf{X}}$ and $\hat{\mathbf{Y}}$ are continuous. This implies that:
	
	\begin{equation*}
		\sup_{l\geq 1}\sup_{0 \leq t \leq T} \lVert X^{n_l}_{i,C}(t) \rVert < \infty.
	\end{equation*}
	
	Furthermore, we have, for all $x \in \N^{\varepsilon \times 2^K-1}$:

	\begin{equation*}
		\lim_{l \rightarrow \infty} R_{i,C}^{n_l}(x) = \frac{\lvert C \rvert}{\sum_{k,U} \ell^\varepsilon(a_k, a_i)\lvert C \cap U \rvert x_{k,U}} \equiv R_{i,C}(x).
	\end{equation*}
	
	Combining all these results, we can take the limits as $l$ goes to infinity in~\eqref{eq:hat} and use the dominated convergence theorem to get the following equation for $(\hat{\mathbf{x}}_\varepsilon, \hat{\mathbf{Y}})$:

	\begin{align*}
		&\hat{x}_{i,C}(t) = \hat{Y}_{i,C}(t) - \frac{1}{L_C}\int_0^t \frac{\lvert C \rvert \hat{x}_{i,C}(s)}{\sum_{k,U}\ell^\varepsilon(a_k, a_i) \lvert C \cap U \rvert \hat{x}_{k,U}(s)} \diff s.
	\end{align*}

	Let us remind that, from Lemma~\ref{lem:convergence}, $\bar{Y}_{i,C}(t) = \bar{X}_{i,C}(0) + \lambda p_C \varepsilon^2 t$, which gives us the following equation for the limit process:


	\begin{multline}\label{eq:fluid}
		\hat{x}_{i,C}(t) = \hat{x}_{i,C}(0) + \lambda p_C \varepsilon^2t - \int_0^t \frac{1}{L_C} \frac{\lvert C \rvert\hat{x}_{i,C}(s)}{\sum_{k,U}\ell^\varepsilon(a_k, a_i) \lvert C \cap U \rvert \hat{x}_{k,U}(s)} \diff s.
	\end{multline}

	To obtain the intended equations, we differentiate~\eqref{eq:fluid}, which concludes the proof.

\end{proof}

\section*{B: Proof of Theorem~\ref{thm:forward}}\label{sec:AppendixB}

In this section, we provide a proof for Theorem~\ref{thm:forward} using a lemma from~\cite{shneer20} describing the evolution of dynamics encompassing those of our fluid model. Let us state the following lemma:

\begin{lem}\label{lem:shneer}
	Assume that there exists $\mathbf{z} \in \R_{+, \star}^{N_\varepsilon \times 2^K -1}$ such that for all $i, C$:
	
	\begin{equation*}
		\lambda p_C \varepsilon^2 \leq \frac{1}{L_C} \frac{\lvert C \rvert z_{i,C}}{\sum_{k, U} \ell^\varepsilon(a_k, a_i) \lvert C \cap U \rvert z_{k,U}}
	\end{equation*}
	
	Then, the Markov chain $\bar{\mathbf{X}}_\varepsilon$ is positive Harris recurrent.
\end{lem}

\begin{proof}
	
	To obtain this result, we use Lemma 2 from~\cite{shneer20}: the fluid limit $\mathbf{x}(t)$ meets the intended requirements. Let us take such a $\mathbf{z} \in \R_{+, \star}^{N_\varepsilon \times 2^K -1}$. We know that for all $0 < \delta < M < \infty$, there exists $T > 0$ such that whenever $\lVert \mathbf{x}(0) \rVert = M$, we have $\lvert \mathbf{x}(T) \rVert < \delta$.

	Theorem 4.2 from~\cite{dai95} with the fluid limit $\bar{\mathbf{x}}$ allows us to obtain positive Harris recurrence for the chain $\mathbf{X}$.
\end{proof}

We use this result in $\R_+^{N_\varepsilon\times 2^K-1}$ by setting $\psi_{i,C}(\mathbf{x}) = \frac{1}{L_C}\frac{\lvert C \rvert x_{i,C}}{\sum_{k,U} \ell^\varepsilon(a_k, a_i) \lvert C \cap U \rvert x_{k,U}}$, which is 0-homogeneous, and non-increasing in $x_{k,U}$ for $(k,U) \neq (i,C)$.

The goal here is to find an appropriate vector $z$ to get the desired upper bound on $\lambda$. Let us set $z_{i,C} = p_C L_C$ for all $C\in \Pk$. We want that, for each $C \in \Pk$:

\begin{equation}\label{eq:ub_lambda}
	\lambda p_C \varepsilon^2 \leq \frac{1}{L_C}\frac{\lvert C \rvert p_C L_C}{\sum_{k,U} \ell^\varepsilon(a_k, a_i) \lvert C \cap U \rvert p_U L_U}.
\end{equation}

Using the torus property in our system, \eqref{eq:ub_lambda} becomes:

\begin{equation}\label{eq:ub_lambda_2}
	\lambda \leq \frac{1}{\langle \ell^\varepsilon_\D \rangle}\frac{\lvert C \rvert}{\sum_U \lvert C \cap U \lvert p_U L_U}.
\end{equation}

We will use the following Lemma to simplify~\eqref{eq:ub_lambda_2}:

\begin{lem}\label{lem:symmetric}
	Let $(y_C)_{C \in \Pk} \in \R^{2^K-1}$ be such that $\lvert C \rvert = \lvert D \rvert \Rightarrow y_C = y_D$. We have, for all $C \in \Pk$:

	\begin{equation}
		\sum_{U \in \Pk}\lvert C \cap U \rvert y_U = \frac{\lvert C \rvert}{K}\sum_{U \in \Pk} \lvert U \rvert y_U.
	\end{equation}
\end{lem}

\begin{proof}

Let $C \in \Pk$ with $\lvert C \rvert = j$, and $(y_U)_{U \in \Pk} \in \R^{2^K-1}$ be such that $y_U = y_V$ whenever $\lvert U \rvert = \lvert V \rvert$. We have:

\begin{equation}\label{eq:part}
	\sum_{U \in \Pk} \lvert C \cap U \rvert y_U = \sum_{l=1}^K \sum_{m=1}^{j\wedge l} \sum_{\substack{U: \lvert U \rvert = l \\ \lvert C \cap U \rvert = m}} m y_U.
\end{equation}

The number of sets $U \in \Pk$ such that $\lvert U \rvert = l$ and $\lvert C \cap U \rvert = m$ is equal to $\binom{j}{m}\binom{K-j}{l-m}$. Using the symmetry property of $y$, we know that $y_U = y_{[1,l]}$. Using the fact that we have $y_U = \frac{1}{\binom{K}{l}}\sum_{\lvert V \rvert = l} y_V(t)$. We can rewrite~\eqref{eq:part} as:

	\begin{align*}
		\sum_{U \in \Pk} \lvert C \cap U \rvert y_U &= \sum_{l=1}^K y_{[1,l]} \sum_{m=1}^{j \wedge l} m \binom{j}{m}\binom{K-j}{l-m} \nonumber \\
		&= \sum_{l =1}^K \left(\sum_{\lvert U \rvert = l} y_U \right) \sum_{m=1}^{j \wedge l} m \frac{\binom{j}{m}\binom{K-j}{l-m}}{\binom{K}{l}}.
	\end{align*}

	We can prove that $\sum_{m=1}^{j \wedge l} m \frac{\binom{j}{m}\binom{K-j}{l-m}}{\binom{K}{l}} = \frac{lj}{K}$ by using the formula for the expectation of a hypergeometric variable with parameters $K$, $j$ and $l$. Replacing $l$ by $\lvert U \rvert$ and $j$ by $\lvert C \rvert$ leads to:

	\begin{equation*}
		\sum_{U \in \Pk}\lvert C \cap U \rvert y_U = \frac{\lvert C \rvert}{K}\sum_{U \in \Pk} \lvert U \rvert y_U,
	\end{equation*}
	
	\noindent which concludes the proof of Lemma~\ref{lem:symmetric}.

\end{proof}

The vector $(p_C)_{C \in \Pk}$ verifies the conditions of Lemma~\ref{lem:symmetric}, which allows us to obtain, after setting $\mathfrak{L} = \sum_U \lvert U \rvert p_U L_U$:

\begin{equation}\label{eq:ub_lambda_3}
	\lambda \leq \frac{1}{\langle \ell^\varepsilon_\D \rangle}\frac{K}{\mathfrak{L}}.
\end{equation}

The result of Lemma~\ref{lem:shneer} gives us stability for the fluid model from Theorem~\ref{thm:fluid_equation} in the sense of Definition 4.1 of~\cite{dai95} with $M = 1$. Applying the result from Theorem 4.2 of the same paper gives us positive Harris recurrence for the chain $\bar{\mathbf{X}}_\varepsilon$. $\bar{\mathbf{X}}_\varepsilon$ is an irreducible Markov jump process that is also positive recurrent. Thus, $\bar{\mathbf{X}}_\varepsilon$ is ergodic, which concludes the proof of Theorem~\ref{thm:forward}.

\section*{C: Proof of Theorem~\ref{thm:reciprocal}}

To prove the instability of the chain $\underline{\mathbf{X}}_\varepsilon$, we start by introducing, for each $i, C$, the function $r_{i,C}$ defined as:
	
	\begin{equation*}
		r_{i,C}(\mathbf{x}) = \frac{1}{p_C L_C \varepsilon^2} \frac{\lvert C \rvert x_{i,C}}{\sum_{k,U} \ell_\varepsilon(a_k, a_i) \lvert C \cap U \rvert x_{k,U}}.
	\end{equation*}

	For all $i, C$, $r_{i,C}$ is continuous and 0-homogenous, and it is not defined for $\mathbf{x} = 0$: let us take the sequences $\mathbf{x}_n$ and $\mathbf{y}_n$ such that $x_{0, \{1\}, n} = y_{0, \{1\}, n} = \frac{1}{n}$, $y_{1, \{1\}, n} = \frac{1}{n}$ and all the other coordinates set to 0. Immediately, $\lim_{n\rightarrow \infty} \mathbf{x}_n = \lim_{n\rightarrow \infty} \mathbf{y}_n = 0$. Using the 0-homogeneity of the functions, we have, for all $n$:
	
	\begin{align*}
		&r_{0, \{1\}}(\mathbf{x}_n) = \frac{1}{p_{\{1\}} L_{\{1\}} \varepsilon^2}, \\
		&r_{0, \{1\}}(\mathbf{y}_n) = \frac{1}{p_{\{1\}} L_{\{1\}} \varepsilon^2} \frac{1}{1 + \ell_\varepsilon(a_0, a_1)},
	\end{align*}
	
	\noindent which have different limits as $n$ goes to infinity.

	Let us now consider the function $\mathbf{x} \in \R_{+,*}^{N_\varepsilon \times 2^K-1} \mapsto \min_{i,C} r_{i,C} (\mathbf{x})$. It is also 0-homogenous, not defined at $\mathbf{x} = 0$, and continuous on the set $\mathcal{S} = \{\mathbf{x} \in \R_{+,*}^{N_\varepsilon \times 2^K-1}: \lvert \mathbf{x} \rvert = 1\}$, which is compact. Thus, it admits a maximum on $\mathcal{S}$, and on $\R_{+,*}^{N_\varepsilon \times 2^K-1}$ as a consequence. We set

	\begin{equation*}
		\mathcal{S} = \argmax_{\mathbf{x} \in \R_{+, *}^{N_\varepsilon \times 2^K-1}}\min_{i, C} \frac{1}{p_C L_C \varepsilon^2} \frac{\lvert C \rvert x_{i,C}}{\sum_{k, U} \ell_\varepsilon(a_k, a_i) \lvert C \cap U \rvert x_{k,U}},
	\end{equation*}

	\noindent which is non-empty.

\begin{lem}\label{lem:instability}
	Let $\mathbf{z} \in \mathcal{S}$ and let $i^\star, C^\star$ be the coordinates of a point where the maximum is attained. Let $\lambda > 0$ be such that:

	\begin{equation*}
		\lambda p_{C^\star} \varepsilon^2 > \frac{1}{L_{C^\star}} \frac{\lvert C ^\star \rvert z_{i^\star, C^\star}}{\sum_{k,U}\lvert C \cap U \rvert \ell_\varepsilon(a_k, a_i^\star) z_{k,U}}.
	\end{equation*}

	Then, $\underline{\mathbf{X}}_\varepsilon$ is transient.

\end{lem}

\begin{proof}
	
	Let $\mathbf{z} \in \mathcal{S}$. By definition, there exists $i^\star, C^\star$ such that the maximum value of $\min_{i,C} r_{i,C}$ is equal to $r_{i^\star, C^\star}(\mathbf{z})$. Let us take $\lambda > 0$ such that:

	\begin{equation*}
		\lambda p_{C^\star} \varepsilon^2 > \frac{1}{L_{C^\star}} \frac{\lvert C ^\star \rvert z_{i^\star, C^\star}}{\sum_{k,U}\lvert C \cap U \rvert \ell_\varepsilon(a_k, a_i^\star) z_{k,U}}.
	\end{equation*}
	
	We define the process $\underline{\mathbf{Y}}$ such that:
	\begin{itemize}
		\item Arrivals in queue $i, C$ happen with rate $\lambda p_C \varepsilon^2$, for all $i, C$;
		\item Departures in queue $i, C$ happen with rate:
		\begin{itemize}
			\item $\frac{1}{L_{C^\star}} \frac{\lvert C ^\star \rvert z_{i^\star, C^\star}}{\sum_{k,U}\lvert C \cap U \rvert \ell_\varepsilon(a_k, a_i) z_{k,U}}$ if $i, C = i^\star, C^\star$,
			\item $\frac{1}{L_C} \frac{\lvert C \rvert \underline{Y}_{i, C}(t)}{\sum_{k,U}\lvert C \cap U \rvert \ell_\varepsilon(a_k, a_i) \underline{Y}_{k,U}(t) + \mathcal{N}_0}$ else.
		\end{itemize}
	\end{itemize}

	$\underline{\mathbf{Y}}$ is a Markov jump process, with state space $\N^{N_\varepsilon \times 2^K -1}$. By definition of $\mathbf{z}$, we have, at all times $t \geq 0$:


	\begin{multline*}
		\frac{1}{L_{C^\star}} \frac{\lvert C ^\star \rvert z_{i^\star, C^\star}}{\sum_{k,U}\lvert C \cap U \rvert \ell_\varepsilon(a_k, a_i) z_{k,U}} \geq \frac{1}{L_{C^\star}} \frac{\lvert C ^\star \rvert \underline{X}_{i^\star, C^\star}(t)}{\sum_{k,U}\lvert C \cap U \rvert \ell_\varepsilon(a_k, a_i) \underline{X}_{k,U}(t) + \mathcal{N}_0}.
	\end{multline*}

	Using a coupling argument similar to the one used in the proof of Theorem~\ref{thm:domination}, we obtain, for all $i, C$ and at all times $t$:

	\begin{equation*}
		\underline{Y}_{i,C}(t) \leq \underline{X}_{i,C}(t), \quad \Proba\textrm{-a.s.}
	\end{equation*}

	The queue $\underline{Y}_{i^\star, C^\star}$ is an M/M/1 queue with constant arrival rate $\lambda p_{C^\star} \varepsilon^2$ and a departure rate equal to $\frac{1}{L_{C^\star}} \frac{\lvert C ^\star \rvert z_{i^\star, C^\star}}{\sum_{k,U}\lvert C \cap U \rvert \ell_\varepsilon(a_k, a_i) z_{k,U}}$. 
	
	It is unstable and $\Proba\left[\lim_{t \rightarrow \infty}\underline{Y}_{i^\star, C^\star}(t) = +\infty\right] = 1$, implying the same for $\underline{X}_{i^\star, C^\star}$, which concludes the proof.

\end{proof}

We can note that queue $\underline{X}_{i^\star, C^\star}$ is not the only queue whose population diverges to infinity. The departure rates in all queues (except for queue $i^\star, C^\star$) are decreasing functions of $\underline{X}_{i^\star, C^\star}(t)$. Thus, for a given queue $\underline{X}_{j,D}$, there exists a time $T_{j,D}$ after which the departure rate in this queue becomes lower than its arrival rate. We can then bound from below $\underline{X}_{j,D}$ after time $T_{j,D}$ by an adequate M/M/1 queue and obtain that the population in queue $j, D$ goes to infinity, $\Proba$-almost surely for all $j, D$.

To complete the proof of Theorem~\ref{thm:reciprocal}, we characterize the value of $\mathbf{z}$:

\begin{lem}\label{lem:system}
	Let $\mathbf{z}\in \mathcal{S}$. Then, $\mathbf{z}$ is a solution to the following system of equations, for all $0 \leq i, j \leq N_\varepsilon-1$ and $C, D \in \Pk$:

	
	\begin{multline}\label{eq:system_0}
		\frac{1}{p_C L_C} \frac{\lvert C \rvert z_{i,C}}{\sum_{k,U} \ell_\varepsilon(a_k, a_i) \lvert C \cap U \rvert z_{k,U}} = \frac{1}{p_D L_D} \frac{\lvert D \rvert z_{j,D}}{\sum_{k,U} \ell_\varepsilon(a_k, a_j) \lvert D \cap U \rvert z_{k,U}}.
	\end{multline}

\end{lem}

\begin{proof}

Let $\mathbf{z}$ be such that:

\begin{equation*}
	\mathbf{z} \in \argmax_{\mathbf{y} \in \R_+^{N_\varepsilon\times 2^K-1}} \min_{i,C} r_{i,C}(\mathbf{y}).
\end{equation*}

To prove this result, we use the maximality of $\mathbf{z}$. Without loss of generality, let us assume that the minimum over $i,C$ of $r_{i,C}(\mathbf{z})$ is reached for $i = 0$ and $C = \{1\}$. We know that the function $r_{0, \{1\}}$ is a decreasing function of $z_{i, C}$ for all $1 \leq i \leq N-1$ and $C \neq \{0\}$. Hence, these values have to be minimal in order to maximize $r_{0, \{1\}}$.

By definition, for all $1 \leq j \leq N_\varepsilon - 1$, $C \neq \{1\}$, we have $r_{0,\{1\}}(\mathbf{z}) \leq r_{i,C}(\mathbf{z})$.
But, for $i \geq 1$ and $C \neq \{0\}$, $r_{i, C}$ is an increasing function of $z_{i, C}$. Thus, the only possible value for $\mathbf{z}$ is such that:

\begin{equation*}
	r_{0,\{1\}}(\mathbf{z}) = r_{i, C}(\mathbf{z}),
\end{equation*}

\noindent which concludes the proof.

\end{proof}

Finally, we know that $\mathbf{z}$ such that for all $i, C$, $z_{i,C} = p_C L_C$ is a solution to~\eqref{eq:system_0}. Using Lemma~\ref{lem:instability}, the system is unstable if:

\begin{align*}
	\lambda p_{C^\star} \varepsilon^2 \geq& \frac{1}{L_{C^\star}} \frac{\lvert C^\star \rvert p_{C^\star}L_{C^\star}}{\sum_{k,U}\ell_\varepsilon(a_k, a_{i^\star}) \lvert C^\star \cap U \rvert p_U L_U} \\
	\overset{(a)}{=}& \frac{\lvert C^\star \rvert p_{C^\star}}{\frac{1}{\varepsilon^2}\langle \ell_{\varepsilon, \D} \rangle \frac{\lvert C^\star \rvert}{K} \sum_U \lvert U \rvert p_U L_U} \\
	=& \frac{\varepsilon^2 p_{C^\star} K}{\langle \ell_{\varepsilon, \D} \rangle \frac{\lvert C^\star \rvert}{K} \sum_U \lvert U \rvert p_U L_U},
\end{align*}

\noindent where $(a)$ uses Lemma~\ref{lem:symmetric} with the vector $(p_C L_C)_{C \in \Pk}$ and the square torus property of $\D$. Setting $\mathfrak{L} = \sum_U \lvert U \rvert p_U L_U$  and rearranging the equation gives us:

\begin{equation*}
	\lambda \geq \frac{K}{\langle \ell_{\varepsilon, \D} \rangle \mathfrak{L}},
\end{equation*}

\noindent which is the intended condition for $\lambda$ and concludes the proof.

\end{document}